\RequirePackage{snapshot}
\documentclass[a4paper,12pt,twoside]{article}
\usepackage{xifthen}
\usepackage{xparse}

\newboolean{amsstuff}
\newboolean{article}
\newboolean{biber_std}
\newboolean{book}
\newboolean{chess}
\newboolean{danger}
\newboolean{datetime}
\newboolean{depth_1}
\newboolean{depth_2}
\newboolean{depth_3}
\newboolean{emptypage}
\newboolean{enumitem}
\newboolean{fouriernc}
\newboolean{frenchspacing}
\newboolean{geometry_a5}
\newboolean{geometry_std}
\newboolean{geometry_a4_full_except_left}
\newboolean{graphicx}
\newboolean{header_std}
\newboolean{hyperref_black}
\newboolean{lang_eng}
\newboolean{lang_ger}
\newboolean{latex}
\newboolean{linespace}
\newboolean{listings}
\newboolean{longtable}
\newboolean{lualatex}
\newboolean{makeidx}
\newboolean{mathpazo}
\newboolean{mathrsfs}
\newboolean{microtype}
\newboolean{mymathfontcommands}
\newboolean{mymathoperators}
\newboolean{mymisccommands}
\newboolean{nofootbreaks}
\newboolean{noparindent}
\newboolean{paralist}
\newboolean{printbib}
\newboolean{printbiburlremark}
\newboolean{qedblack}
\newboolean{raggedbottom}
\newboolean{section_std}
\newboolean{singleletterredefinitions}
\newboolean{squeezedisplaymath}
\newboolean{showlabels}
\newboolean{subsection_arabic}
\newboolean{subsection_Alph}
\newboolean{springer}
\newboolean{stdmargins}
\newboolean{utf8}
\newboolean{utf8x}
\newboolean{t1}
\newboolean{textcomp}
\newboolean{theorem_std}
\newboolean{toc_compressed}
\newboolean{tikzcd}
\newboolean{xelatex}
\newboolean{xy}
\newboolean{youngtab}

\setboolean{amsstuff}{true}
\setboolean{article}{false}
\setboolean{biber_std}{true}
\setboolean{book}{true}
\setboolean{chess}{true}
\setboolean{datetime}{false}
\setboolean{danger}{true}
\setboolean{depth_1}{false}
\setboolean{depth_2}{true}
\setboolean{depth_3}{false}
\setboolean{emptypage}{true}
\setboolean{enumitem}{true}
\setboolean{fouriernc}{false}
\setboolean{frenchspacing}{true}
\setboolean{geometry_std}{true}
\setboolean{header_std}{true}
\setboolean{hyperref_black}{true}
\setboolean{lang_eng}{true}
\setboolean{lang_ger}{false}
\setboolean{latex}{true}
\setboolean{linespace}{false}
\setboolean{listings}{false}
\setboolean{longtable}{true}
\setboolean{lualatex}{false}
\setboolean{makeidx}{false}
\setboolean{mathpazo}{true}
\setboolean{mathrsfs}{true}
\setboolean{microtype}{false}
\setboolean{mymathoperators}{true}
\setboolean{mymathfontcommands}{true}
\setboolean{mymisccommands}{true}
\setboolean{nofootbreaks}{true}
\setboolean{noparindent}{false}
\setboolean{paralist}{false}
\setboolean{printbib}{false}
\setboolean{printbiburlremark}{false}
\setboolean{qedblack}{true}
\setboolean{raggedbottom}{true}
\setboolean{section_std}{true}
\setboolean{squeezedisplaymath}{true}
\setboolean{showlabels}{false}
\setboolean{springer}{true}
\setboolean{singleletterredefinitions}{false}
\setboolean{subsection_arabic}{false}
\setboolean{subsection_Alph}{true}
\setboolean{stdmargins}{true}
\setboolean{utf8}{true}
\setboolean{utf8x}{false}
\setboolean{t1}{true}
\setboolean{textcomp}{true}
\setboolean{theorem_std}{true}
\setboolean{toc_compressed}{true}
\setboolean{tikzcd}{true}
\setboolean{xelatex}{false}
\setboolean{xy}{false}
\setboolean{youngtab}{false}
\setboolean{biber_std}{true}
\setboolean{printbib}{true}
\setboolean{depth_3}{false}
\setboolean{depth_2}{true}
\setboolean{depth_1}{false}
\setboolean{article}{true}
\setboolean{book}{false}
\setboolean{microtype}{true}
\setboolean{printbiburlremark}{true}
\setboolean{linespace}{false}            
\usepackage{caption}

\usepackage{float}
\restylefloat{table}

\ifthenelse{\boolean{listings}}
{
	\usepackage{listings}
	\lstset{basicstyle=\footnotesize\ttfamily} 
}
{}

\ifthenelse{\boolean{chess}}
{
	\usepackage{xskak}
	\usepackage{caption}
	\captionsetup[table]{font=footnotesize,labelfont=bf,labelsep=none,margin=0pt,skip=0pt,name=Nr.}

\newcommand{\chessboardwithcaptions}[9]{%
\begin{figure}[#1]
{\centering

\ifx&#2&\else \captionof{table}{\label{#2}}\fi

\ifx&#4&\else \noindent #4 \fi

\ifx&#5&\else \noindent #5 \fi

\ifx&#6&\else \noindent {\footnotesize #6} \fi

\begin{tabular}{lr}
\multicolumn{2}{c}{\chessboard[marginwidth=0pt,margintopwidth=0.2em,label=false,showmover,moversize=0.5em, #3]} \\
\ifx&#6&\ifx&#8&\else {\footnotesize #7} & {\footnotesize #8} \\ \fi \else {\footnotesize #7} & {\footnotesize #8} \\ \fi
\ifx&#9&\else \multicolumn{2}{c}{{\footnotesize #9}} \\ \fi
\end{tabular}

}
\end{figure}
}

}
{}

\ifthenelse{\boolean{toc_compressed}}
{
    \usepackage{titletoc}%
    \ifthenelse{\boolean{book}}
    {
        \titlecontents{section}
        [3em]
        {}
        {\contentslabel[\hfill\textsection\thecontentslabel.\hspace{0.3em}]{3em}}
        {}
        {\titlerule*[0.5pc]{.}\contentspage}
 
        \titlecontents*{subsection}%
          [3em]%
          {\small}%
          {\footnotesize\textsection\thecontentslabel.\hspace{0.3em}}%
          {}%
          {\footnotesize\hspace{0.3em}\thecontentspage}%
          [,\ ]%
          []%
      }
      {}

      \ifthenelse{\boolean{article}}
    {
        \titlecontents{section}
        [3em]
        {}
        {\bfseries\contentslabel[\hfill\textsection\thecontentslabel.\hspace{0.3em}]{3em}}
        {}
        {\bfseries\titlerule*[0.5pc]{.}\contentspage}
 
        \titlecontents*{subsection}%
          [3em]%
          {\small}%
          {\footnotesize\textsection\thecontentslabel.\hspace{0.3em}}%
          {}%
          {\footnotesize\hspace{0.3em}\thecontentspage}%
          [,\ ]%
          []%
      }
      {}
}
{}

\ifthenelse{\boolean{linespace}}
{
    \usepackage{setspace}
    \setstretch{1.1}
}
{}

\ifthenelse{\boolean{springer}}
{
    \usepackage{color}
    \definecolor{shadecolor}{gray}{0.90}

   \makeatletter
       \def\namedlabel#1#2{\begingroup
        #2%
        \def\@currentlabel{#2}%
        \phantomsection\label{#1}\endgroup
    }
    \makeatother
}
{}

\ifthenelse{\boolean{lualatex}}
{
    \ifthenelse{\boolean{mathpazo}}
    {
        \usepackage{fontspec}
        \setmainfont[Ligatures=TeX, Numbers=OldStyle]{TeX Gyre Pagella}
    }
    {}
}
{}

\ifthenelse{\boolean{textcomp}}
{
    \usepackage{textcomp}
}
{}

\ifthenelse{\boolean{makeidx}}
{
	\usepackage{makeidx}
	\makeindex
	
    \newcommand{\word}[1]{\textit{#1}\index{#1}} 
    
    \newcommand{\invword}[1]{\index{#1}}
    
}
{
    \newcommand{\word}[1]{\textit{#1}} 
    
    \newcommand{\invword}[1]{}
    
}

\ifthenelse{\boolean{microtype}}
{
    \usepackage[protrusion=true,expansion=true,babel=true,final]{microtype}
}
{}

\ifthenelse{\boolean{graphicx}}
{
    \ifthenelse{\boolean{latex}}
    {
        \usepackage[pdftex]{graphicx}
    }
}
{}

\ifthenelse{\boolean{noparindent}}
{
    \setlength{\parindent}{0pt}
}
{}

\ifthenelse{\boolean{lang_eng}}
{
    \usepackage[american]{babel}
}
{}

\ifthenelse{\boolean{lang_ger}}
{
    \usepackage[ngerman]{babel}
}
{}

\ifthenelse{\boolean{longtable}}
{
    \usepackage{longtable}
}
{}

\ifthenelse{\boolean{paralist}}
{
    \usepackage{paralist}
}
{}

\ifthenelse{\boolean{emptypage}}
{
    \usepackage{emptypage}
}
{}

\ifthenelse{\boolean{datetime}}
{
    \usepackage{datetime}
}
{}

\ifthenelse{\boolean{utf8} \and \boolean{latex}}
{
    \usepackage[utf8]{inputenc}
}
{}

\ifthenelse{\boolean{utf8x} \and \boolean{latex}}
{
    \usepackage[utf8x]{inputenc}
}
{}

\ifthenelse{\boolean{t1} \and \boolean{latex}}
{
    \usepackage[T1]{fontenc}
}
{}

\ifthenelse{\boolean{mathrsfs}}
{
    \usepackage{mathrsfs}
}
{}

\ifthenelse{\boolean{amsstuff}}
{
    \usepackage{amssymb}
    \usepackage{amsmath}
    \usepackage{amsthm}
}
{}

\ifthenelse{\boolean{fouriernc} \and \not\boolean{lualatex}}
{
    \usepackage{fourier}
    \DeclareMathAlphabet{\mathcal}{OMS}{cmsy}{m}{n} %
    \usepackage{tgschola}
}
{}

\ifthenelse{\boolean{mathpazo} \and \not\boolean{lualatex}}
{
    \usepackage[sc,osf]{mathpazo}
}
{}

\ifthenelse{\boolean{nofootbreaks}}
{
    \interfootnotelinepenalty=10000
}

\ifthenelse{\boolean{qedblack}}
{
    \renewcommand{\qedsymbol}{$\blacksquare$}
}
{}

\ifthenelse{\boolean{frenchspacing}}
{
    \frenchspacing
}
{}

\ifthenelse{\boolean{raggedbottom}}
{
    \raggedbottom
}
{}

\ifthenelse{\boolean{hyperref_black}}
{
    \usepackage[hyphens]{url}
    \ifthenelse{\boolean{latex}}
    {
        \usepackage[pdftex]{hyperref}
    }
    {
        \usepackage{hyperref}
    }
    
    \hypersetup
    {
        colorlinks=true,
        linkcolor=black,
        urlcolor=black,
        filecolor=black,
        citecolor=black
    }
}{}

\ifthenelse{\boolean{enumitem}}
{
    \usepackage{enumitem}

\newlist{enum_thm}{enumerate}{1}
\setlist[enum_thm,1]
{
    label=(\alph*),
    noitemsep,
    nolistsep,
    align=left,
    labelindent=\parindent,
    leftmargin=*
}

\newlist{enum_proof}{enumerate}{1}
\setlist[enum_proof,1]
{
    label=(\alph*),
    noitemsep,
    nolistsep,
    align=left,
    labelindent=\parindent,
    itemindent=*,
    leftmargin=0pt
    }

}
{}
\ifthenelse{\boolean{geometry_std}}
{
    \usepackage[pdftex]{geometry}
    
    \geometry{tmargin=35mm,bmargin=30mm,lmargin=27mm,rmargin=27mm}
    \setlength\headsep{2ex}
    \setlength\marginparsep{1em}
    \setlength\marginparwidth{2.5em}
    \setlength\marginparpush{0mm}
}{}

\ifthenelse{\boolean{geometry_a4_full_except_left}}
{
    \usepackage[pdftex]{geometry}
    
    \geometry{tmargin=15mm,bmargin=10mm,lmargin=17mm,rmargin=5mm}
    \setlength\headsep{2ex}
    \setlength\marginparsep{1em}
    \setlength\marginparwidth{2.5em}
    \setlength\marginparpush{0mm}
}{}

\ifthenelse{\boolean{geometry_a5}}
{
    \usepackage[pdftex]{geometry}
    
    \geometry{tmargin=20mm,bmargin=10mm,lmargin=20mm,rmargin=20mm}
    \setlength\headsep{1.5ex}
    \setlength\marginparsep{0.4em}
    \setlength\marginparwidth{1.2em}
    \setlength\marginparpush{0mm}
}{}

\ifthenelse{\boolean{header_std}}
{
    \usepackage{fancyhdr}
    	
    \fancypagestyle{stdstyle}
    {
        \pagestyle{fancyplain}

        \fancyhead{}
        \fancyhead[LO]{\footnotesize\textsection\nouppercase{\rightmark}}
        \fancyhead[RE]{\footnotesize\textsection\nouppercase{\leftmark}}
        \fancyhead[LE,RO]{\footnotesize\thepage}
        \fancyfoot{}
        \renewcommand{\headrulewidth}{0.5pt}
        \renewcommand{\footrulewidth}{0pt}
    }
    
    \fancypagestyle{chapteronly}
    {
        \pagestyle{stdstyle}
        \fancyhead[LO]{}
        \fancyhead[RE]{\footnotesize\nouppercase{\leftmark}}
        \fancyhead[LE,RO]{\footnotesize\thepage}
    }
    
    \fancypagestyle{pagenumonly}
    {
        \pagestyle{stdstyle}
        \fancyhead[LO]{}
        \fancyhead[RE]{}
        \fancyhead[LE,RO]{\footnotesize\thepage}
    }
    
    \fancypagestyle{empty}
    {
        \pagestyle{fancyplain}
        \fancyhead[LO,RE,LE,RO]{}
        \renewcommand{\headrulewidth}{0pt}
        \renewcommand{\footrulewidth}{0pt}
    }

    \pagestyle{stdstyle}
}
{}

\ifthenelse{\boolean{biber_std}}
{
    \usepackage[backend=biber, hyperref=true, language=auto, texencoding=auto, bibencoding=auto, isbn=false, doi=true, style=alphabetic, maxnames=5, minnames=1, sorting=anyt]{biblatex}
    \DeclareFieldFormat{postnote}{#1}
    \DeclareFieldFormat{multipostnote}{#1}
    \renewcommand{\bibfont}{\normalfont\footnotesize}
    \DeclareFieldFormat
  [article,inbook,incollection,inproceedings,patent,thesis,unpublished]
  {title}{\it{#1\isdot}}
}
{}

\ifthenelse{\boolean{theorem_std}}
{
    \swapnumbers 

    \newtheoremstyle{std_style_theorem}%
        {8pt}%
        {8pt}%
        {\normalfont}%
        {0pt}%
        {\normalfont\bfseries}%
        {.}%
        { }%
        {}%

    \theoremstyle{std_style_theorem}
    
    \ifthenelse{\boolean{depth_1}}
    {
        \newtheorem{thm}{Theorem}
    }
    {}

    \ifthenelse{\boolean{depth_2}}
    {
        \newtheorem{thm}{Theorem}[section]
    }
    {}

    \ifthenelse{\boolean{depth_3}}
    {
        \newtheorem{thm}{Theorem}[subsection]
    }
    {}

    \ifthenelse{\boolean{lang_eng}}
    {
        \newtheorem{paran}[thm]{}
        \newtheorem{lemma}[thm]{Lemma}
        \newtheorem{prop}[thm]{Proposition}
        \newtheorem{cor}[thm]{Corollary}
        \newtheorem{defn}[thm]{Definition}
        \newtheorem{ass}[thm]{Assumption}
        \newtheorem{remark}[thm]{Remark}
        \newtheorem{ex}[thm]{Example}
         \newtheorem{exercise}[thm]{Exercise}
        \newtheorem{conjecture}[thm]{Conjecture}
        \newtheorem{question}[thm]{Question}
        \newtheorem{convention}[thm]{Convention}
        \newtheorem{notation}[thm]{Notation}
        \newtheorem{aim}[thm]{Aim}
        \newtheorem{refe}[thm]{Reference}
        \newtheorem{prob}[thm]{Problem}
        
        \newtheorem*{solution}{Solution}
        \newtheorem*{ack}{Acknowledgements}
        \newtheorem{fact}[thm]{Fact}
        \newtheorem{algo}[thm]{Algorithm}
    }
    {}
    
    \ifthenelse{\boolean{lang_ger}}
    {
        \newtheorem{paran}[thm]{}
        \newtheorem{lemma}[thm]{Lemma}
        \newtheorem{prop}[thm]{Proposition}
        \newtheorem{cor}[thm]{Korollar}
        \newtheorem{defn}[thm]{Definition}
        \newtheorem{ass}[thm]{Annahme}
        \newtheorem{remark}[thm]{Bemerkung}

        \newtheorem{conjecture}[thm]{Vermutung}
        \newtheorem{question}[thm]{Frage}
        \newtheorem{convention}[thm]{Konvention}

    }
    {}

    \newtheoremstyle{std_style_theorem_nn}%
        {8pt}%
        {8pt}%
        {\normalfont}%
        {\parindent}%
        {\normalfont\bfseries}%
        {}%
        {0pt}%
        {}%
    
    \theoremstyle{std_style_theorem_nn}
    \newtheorem*{para}{}

    \newtheoremstyle{std_style_theorem_ni}%
        {8pt}%
        {8pt}%
        {\normalfont}%
        {0pt}%
        {\normalfont\bfseries}%
        {}%
        {0pt}%
        {}%
    
    \theoremstyle{std_style_theorem_ni}
    \newtheorem*{parani}{}
    
    \newtheoremstyle{std_style_theorem_nospace}%
        {0pt}%
        {0pt}%
        {\normalfont}%
        {0pt}%
        {\normalfont\bfseries}%
        {.}%
        { }%
        {}%

    \usepackage[framemethod=latex]{mdframed}
    \usepackage{color}
    \definecolor{light-gray}{gray}{0.9}

    \newenvironment{ass_boxed}
    {
        \begin{mdframed}[usetwoside=false,leftmargin=\the\parindent,%
            rightmargin=\the\parindent,skipabove=4pt,skipbelow=2pt,%
            backgroundcolor=light-gray]
            \begin{ass}
    }
    {
        \end{ass}\end{mdframed}
    }

    \newenvironment{thm_boxed}
    {
        \begin{mdframed}[usetwoside=false,leftmargin=\the\parindent,%
            rightmargin=\the\parindent,skipabove=4pt,skipbelow=2pt,%
            backgroundcolor=light-gray]
            \begin{thm}
    }
    {
        \end{thm}\end{mdframed}
    }

    \newenvironment{conjecture_boxed}
    {
        \begin{mdframed}[usetwoside=false,leftmargin=\the\parindent,%
            rightmargin=\the\parindent,skipabove=4pt,skipbelow=2pt,%
            backgroundcolor=light-gray]
            \begin{conjecture}
    }
    {
        \end{conjecture}\end{mdframed}
    }

    \usepackage{manfnt}
    \newdimen\dbendwidth
    \setbox0=\hbox{\dbend}
    \dbendwidth=\wd0
    \newdimen\dbendindent
    \setlength{\dbendindent}{2pc}
    \addtolength{\dbendindent}{-\dbendwidth}
     
    \newtheoremstyle{std_style_theorem_danger}%
        {8pt}%
        {8pt}%
        {\normalfont\hangindent=2pc\hangafter=-2}%
        {0pt}%
        {\normalfont\bfseries}%
        {}%
        {0pt}%
        {}%
    \theoremstyle{std_style_theorem_danger}
    
}
{}

{
\ifthenelse{\boolean{section_std}}
{
    \usepackage[noindentafter]{titlesec}

    \ifthenelse{\boolean{book}}
    {
        \renewcommand{\thechapter}{\Roman{chapter}}
        \usepackage{chngcntr}
	\counterwithout{section}{chapter}
    }
    {}
    
    	\ifthenelse{\boolean{lang_eng}}
	{
        \titleformat{\chapter}[display]
        {\normalfont\filcenter\bfseries}
        {Chapter \thechapter}
        {0pc}
        {\Large
        \vspace{3ex}
        \titlerule \vspace{10pt}}
        [\vspace{8pt} \titlerule ]
    }
    {}
    
    \ifthenelse{\boolean{lang_ger}}
    {
        \titleformat{\chapter}[display]
        {\normalfont\filcenter\bfseries}
        {Kapitel \thechapter}
        {0pc}
        {\Large
        \vspace{3ex}
        \titlerule \vspace{10pt}}
        [\vspace{8pt} \titlerule ]
    }
    {}
    
    \titlespacing{\chapter}{0pt}{0pt}{5ex}
    \titlespacing*{name=\chapter,numberless}{0pt}{\baselineskip}{\baselineskip}

    \renewcommand*\thesection{\arabic{section}}
    \titleformat{\section}{\center\normalfont\large\bfseries}{\textsection\thesection.}{.3em}{}
    \titlespacing{\section}{0pt}{14pt}{4pt}

    \titlespacing{\subsection}{0pt}{14pt}{0pt}
    
    \ifthenelse{\boolean{subsection_Alph}}
    {
        \renewcommand{\thesubsection}{\arabic{section}\Alph{subsection}}
        \titleformat{\subsection}{\center\normalfont\bfseries}{\textsection\thesection\Alph{subsection}.}{.3em}{}
    }
    {}
    \ifthenelse{\boolean{subsection_arabic}}
    {
        \titleformat{\subsection}{\center\normalfont\bfseries}{\thesubsection.}{.3em}{}
    }
    {}

}{}

\ifthenelse{\boolean{xy}}
{
    \usepackage[all]{xy}
    \newdir{ >}{{}*!/-9pt/@{>}}
    
}
{}

\ifthenelse{\boolean{tikzcd}}
{
   \usepackage{tikz-cd}
}
{}

\ifthenelse{\boolean{mymathoperators}}
{
    \DeclareMathOperator{\GL}{GL}
    \DeclareMathOperator{\SL}{SL}

    \DeclareMathOperator{\Ker}{Ker}
    
    \let\Im\undefined
    \DeclareMathOperator{\Im}{Im}

    \let\dim\undefined
    \DeclareMathOperator{\dim}{dim}

    \DeclareMathOperator{\id}{id}

}
{}

\ifthenelse{\boolean{mymisccommands}}
{
    \newcommand{\subs}{\subseteq}
    
    \ifthenelse{\boolean{lualatex}}
    {
        \let\oldcommandsups\sups
        \let\sups\undefined        \newcommand{\sups}{\ifthenelse{\boolean{mmode}}{\subseteq}{\oldcommandsups}}
    }
    {  
        \newcommand{\sups}{\supseteq}
    }
    
    \newcommand{\dopgleich}{\mathrel{\mathop:}=}
    \newcommand{\gleichdop}{=\mathrel{\mathop:}}
    
    \newcommand{\ol}[1]{\overline{#1}}

    \newcommand{\wt}[1]{\widetilde{#1}}

    \newcommand{\rarr}{\rightarrow}

    \newcommand{\eps}{\varepsilon}

    \DeclareDocumentCommand{\cat}{m g g}
    {
        \IfNoValueTF{#2}
        {
            \IfNoValueTF{#3}
            {
                \mathsf{#1}
            }
            {
                \ensuremath{\mathsf{#1}_{#3}}
            }
        }
        {
            \IfNoValueTF{#3}
            {
                \ensuremath{{_{#2}}\mathsf{#1}}
            }
            {
                \ensuremath{{_{#2}}\mathsf{#1}_{#3}}
            }
        }
    }

    \newcommand{\name}[1]{\textsc{#1}}  
    
    \newcounter{commcount}

}
{}

\ifthenelse{\boolean{mymathfontcommands}}
{

    \newcommand{\mbi}[1]{\textbf{\em #1}}
    \newcommand{\mrm}[1]{\mathrm{#1}}

    \newcommand{\tn}[1]{\textnormal{#1}}

    \newcommand{\biq}{\mbi{q}}

    \newcommand{\biu}{\mbi{u}}

    \newcommand{\fm}{\mathfrak{m}}

    \newcommand{\fp}{\mathfrak{p}}

    \newcommand{\rt}{\mathrm{t}}

    \newcommand{\rG}{\mathrm{G}}
    \newcommand{\rH}{\mathrm{H}}
    \newcommand{\rI}{\mathrm{I}}

    \newcommand{\rL}{\mathrm{L}}
    \newcommand{\rM}{\mathrm{M}}

    \newcommand{\rP}{\mathrm{P}}

    \newcommand{\rS}{\mathrm{S}}
    \newcommand{\rT}{\mathrm{T}}

    \newcommand{\rZ}{\mathrm{Z}}

    \newcommand{\sF}{\mathscr{F}}
    
    \newcommand{\sH}{\mathscr{H}}

    \newcommand{\sO}{\mathscr{O}}

    \newcommand{\bbC}{\mathbb{C}}

    \newcommand{\bbF}{\mathbb{F}}

    \newcommand{\bbN}{\mathbb{N}}
    
    \newcommand{\bbP}{\mathbb{P}}
    \newcommand{\bbQ}{\mathbb{Q}}
    \newcommand{\bbR}{\mathbb{R}}

    \newcommand{\bbZ}{\mathbb{Z}}
    
}
{}

\usepackage{comment}

\captionsetup[table]{labelsep=period, name=Table}

\begin{document}

\ifthenelse{\boolean{squeezedisplaymath}}
{
    \setlength{\abovedisplayshortskip}{\the\smallskipamount}%
    \setlength{\belowdisplayshortskip}{\the\smallskipamount}%
    \setlength{\abovedisplayskip}{\the\smallskipamount}%
    \setlength{\belowdisplayskip}{\the\smallskipamount}%
}
{}
\author{Ulrich Thiel\footnote{University of Stuttgart, \texttt{thiel at mathematik.uni-stuttgart.de}}}
\title{A counter-example to Martino's conjecture about generic Calogero--Moser families}
\maketitle
\pagestyle{pagenumonly}
\thispagestyle{empty}
\vspace{-12pt}
\begin{abstract}
The Calogero--Moser families are partitions of the irreducible characters of a complex reflection group derived from the block structure of the corresponding restricted rational Cherednik algebra. It was conjectured by Martino in 2009 that the generic Calogero--Moser families coincide with the generic Rouquier families, which are derived from the corresponding Hecke algebra. This conjecture is already proven for the whole infinite series $G(m,p,n)$ and for the exceptional group $G_4$. 
A combination of theoretical facts with explicit computations enables us to determine the generic Calogero--Moser families for the nine exceptional groups $G_4$, $G_5$, $G_6$, $G_8$, $G_{10}$, $G_{23}=H_3$, $G_{24}$, $G_{25}$, and $G_{26}$. We show that the conjecture holds for all these groups---except surprisingly for the group $G_{25}$, thus being the first and only-known counter-example so far.
\end{abstract}

{\centering
\begin{minipage}{0.87\textwidth}
{\small
\noindent \textbf{Remark.} The final version will appear in  \mbox{\textit{Algebras and Representation theory}} and is already available online at \url{http://dx.doi.org/10.1007/s10468-013-9449-4}.
}
\end{minipage}

}

\tableofcontents

\newpage
\section{Introduction}

\begin{parani}
For a complex reflection group $\Gamma \dopgleich (W,V)$ Etingof--Ginzburg \cite{EG-Symplectic-reflection-algebras} have introduced a family of $\bbC$-algebras, called \word{rational Cherednik algebras}. %
Starting with Gordon \cite{Gor-Baby-verma} interest in a certain canonical finite-di\-men\-sion\-al quotient, the so-called \word{restricted rational Cherednik algebra}, of the rational Cherednik algebra arose. The irreducible modules of these algebras are parametrized by the irreducible modules of $W$ and so the block structure yields a partition of them whose members are called the \word{Calogero--Moser families} of $\Gamma$. These families are known for the whole infinite series $G(m,p,n)$ of complex reflection groups and all parameters due to results by Gordon \cite{Gor-Baby-verma}, Gordon--Martino \cite{GorMar-Calogero-Moser-space-rest-0}, Bellamy \cite{Bel-The-Calogero-Moser-partit-0}, and Martino \cite{Martino.M11Blocks-of-restricted}. Furthermore, Bellamy \cite{Bel-Singular-CM} determined them generically for the exceptional group $G_4$ but there is almost nothing  known for the remaining exceptional groups.

Although the Calogero--Moser families are interesting objects per se, it is their  (conjectural) broader context which makes them so important. Namely, the achievement of Gordon--Martino \cite{GorMar-Calogero-Moser-space-rest-0} is much more than the sole determination of the Calogero--Moser families for $G(m,1,n)$ because they also conjectured for a Weyl group the existence of a natural bijection between the Calogero--Moser families and the families defined by the Kazhdan--Lusztig cells (the Lusztig families). They have proven this conjecture for type $B_2$ and collected further evidence. This gave a conjectural link between restricted rational Cherednik algebras and Hecke algebras, and so the question arose if there are further, more general, connections to Hecke algebras. The hope of investigating such connections is on the one hand to gain new information about objects attached to Hecke algebras, but on the other hand to generalize constructions from Weyl or Coxeter groups to all complex reflection groups. The main goals are to set up a cell theory for all complex reflection groups---Bonnaf\'e--Rouquier \cite{Bonnafe.C;Rouquier.R13Cellules-de-Calogero} recently made a big step in this direction using rational Cherednik algebras---and to make further progress in the spets program initiated by  Brou\'e--Malle--Michel \cite{Broue.M;Malle.G;Michel.J99Towards-spetses.-I, Broue.M;Malle.G;Michel.JSplit-Spetses-for-pr} whose aim is to find the analogs of finite groups of Lie type for a complex reflection group.
\end{parani}

\begin{para}
To begin studying such broad connections we do not yet need a cell theory, however, since a sensible extension of the Gordon--Martino conjecture to all complex reflection groups is to relate the Calogero--Moser families with the Rouquier families, which are also defined for  any complex reflection group due to Rouquier \cite{Rouquier.R99Familles-et-blocs-da}, Brou\'e--Kim \cite{BroKim02-Familles-de-cara}, Malle--Rouquier \cite{MalRou-Familles-de-caracteres-de-0}, and Chlouveraki \cite{Chl09-Blocks-and-famil}. This idea was already formulated by Gordon--Martino \cite{GorMar-Calogero-Moser-space-rest-0} but Martino \cite{Mar-CM-Rouqier-partition} presented a general conjecture---actually two conjectures: one for arbitrary parameters and a stronger one for generic parameters---relating these two families and proved it for the groups $G(m,1,n)$. Later Bellamy \cite{Bel-The-Calogero-Moser-partit-0} proved it for the whole infinite series $G(m,d,n)$ and generically for $G_4$.%
\end{para}

\begin{para}
In this article we start the first attack on Martino's conjecture for exceptional complex reflection groups. To this end, we essentially just use two theoretical arguments to be discussed in \S2. On the one hand, we introduce the notion of \word{supersingular} characters, which simply emphasizes a result already used by Bellamy \cite{Bel-Singular-CM}. On the other hand, we introduce the notion of \word{Euler families} and the \word{Euler variety} to get a quick approximation of the Calogero--Moser families. The explicit computations in \S3 will prove the main theorems \ref{bad_generic_euler_families_thm} and \ref{main_theorem_2} of this article which state that we can explicitly compute the generic Calogero--Moser families for all the groups
\[
G_4, G_5, G_6, G_8, G_{10}, G_{23} = H_3, G_{24}, G_{25}, G_{26} 
\]
and know some additional properties. A comparison with the generic Rouquier families computed by Chlouveraki \cite{Chlouveraki.MGAP-functions-for-th} shows that Martino's conjecture holds for all these groups---except surprisingly for the group $G_{25}$ which yields the first and only known counter-example to this conjecture. It is not really clear what the abstract reason for this failure is and if there are further counter-examples. This raises an important question: What is the actual relation between Calogero--Moser families and Rouquier families, and do Calogero--Moser families contain additional information about spetses?
\end{para}

\begin{ack}
I would like to thank C\'edric Bonnaf\'e for many discussions and sharing preliminary versions of his joint work \cite{Bonnafe.C;Rouquier.R13Cellules-de-Calogero} with Rapha\"el Rouquier---both leading to major clarifications of my thoughts. Also, I would like to thank Maurizio Martino for many discussions about his conjecture. Furthermore, I would like to thank Gwyn Bellamy, Meinolf Geck, and Gunter Malle for numerous helpful comments on a preliminary version of this article. Last but not least I would like to thank the (anonymous) reviewer of \textit{Algebras and Representation Theory} for a lot of detailed comments and helping me to improve this article. 

I was partially supported by the \textit{DFG Schwerpunktprogramm Darstellungstheorie 1388}. 
\end{ack}

\section{Theoretical basics}

\begin{parani}
In this paragraph we will recall and develop all necessary theoretical ingredients which are then used in the following paragraph to obtain the results.
\end{parani}

\subsection{Reflection groups}

\begin{paran} 
\label{crg_def}
We fix some notations around reflection groups and refer the reader to \cite{BMR-Complex-Reflection} for more details. By a \word{reflection group} we always mean a \word{finite irreducible complex reflection group} $\Gamma \dopgleich (W,V)$. This means that $W$ is a finite group acting faithfully and irreducibly on a finite-dimensional $\bbC$-vector space $V$ such that $W$ is generated by its \word{reflections} in $V$, which are those elements $s \in W$ whose fixed space is a hyperplane in $V$, i.e., $\dim \Ker(s-\id_V) = \dim V-1$. We denote the set of reflections of $\Gamma$ by $\mrm{Ref}(\Gamma)$ and we denote the set of reflection hyperplanes of reflections of $\Gamma$ by $\mrm{Hyp}(\Gamma)$. 

For $s \in \mrm{Ref}(\Gamma)$ we call a non-zero element $\alpha_s \in \Im(\id_V -s )$ a \word{root} of $s$ and we call an element $\alpha_s^\vee$ whose kernel is equal to the \word{reflection hyperplane} $H_s \dopgleich \Ker(\id_V -s)$ of $s$ a \word{coroot} of $s$. If $(\alpha_s,\alpha_s^\vee)$ is a \word{root-coroot pair} of $s$, then 
\[
s(x) = x - (1-\eps_s) \frac{\langle x,\alpha_s^\vee \rangle}{\langle \alpha_s,\alpha_s^\vee \rangle} \alpha_s
\]
for all $x \in V$, where $\eps_s$ is the unique non-trivial eigenvalue of $s$ and $\langle \cdot, \cdot \rangle$ is the canonical pairing between $V$ and its dual space $V^*$. 

The group $W$ acts canonically on both $\mrm{Ref}(\Gamma)$ and $\mrm{Hyp}(\Gamma)$. For $H \in \mrm{Hyp}(\Gamma)$ we denote by $W_H$ the pointwise stabilizer of $H$. The order $e_H$ of $W_H$ is independent of the orbit $\Omega_H$ of $H$ under $W$ and therefore we also denote it by $e_{\Omega_H}$. We now define $\boldsymbol{\Omega}(\Gamma)$ to be the set of all pairs $(\Omega,j)$ with $\Omega \in \mrm{Hyp}(\Gamma)/W$ and $1 \leq j \leq e_\Omega -1$. Then there is a canonical bijection between $\boldsymbol{\Omega}(\Gamma)$ and $\mrm{Ref}(\Gamma)/W$. 
\end{paran}

\begin{paran} \label{shephard_todd}
The reflection groups have been classified up to conjugacy by Shephard--Todd \cite{Shephard.G;Todd.J54Finite-unitary-refle} and were labeled by $G(m,p,n)$ for $m,p,n \in \bbN_{>0}$ satisfying $p \mid m$ and $(m,p,n) \neq (2,2,2)$, and by $G_{i}$ for $4 \leq i \leq 37$, where there are some redundancies in the series $G(m,p,n)$ however.

Explicit representatives of the \word{exceptional groups} $G_i$ for $4 \leq i \leq 37$ can be found for example in \cite{LT-Reflection-groups}, or in the \textsc{GAP3} computer algebra package \textsc{CHEVIE} (see \cite{CHEVIE-JM-4}) by using the command 
\texttt{ComplexReflectionGroup(i).matgens}, or in the computer algebra system \textsc{Magma} (see \cite{Magma-2.18-2}) by using the command \texttt{Shephard\-Todd(i)}. The realizations in \textsc{Magma} were implemented by Taylor and are those given in \cite{LT-Reflection-groups}. 

The character tables of the exceptional groups can quickly be computed in \textsc{Magma}, but they are also available in \textsc{CHEVIE}. The invariant degrees of the exceptional groups are listed in \cite[table D.3]{LT-Reflection-groups} and are also contained in \textsc{CHEVIE}. 
\end{paran}

\subsection{Rational Cherednik algebras}

\begin{paran}
For a reflection group $\Gamma \dopgleich (W,V)$ Etingof--Ginzburg \cite{EG-Symplectic-reflection-algebras} introduced a family of $\bbC$-algebras, called rational Cherednik algebras. Depending on the context different parameters are used for their definition, but all of them can be transformed into each other. We will use the same parameters here as used by Martino \cite{Mar-CM-Rouqier-partition}.

To this end, we choose for each $H \in \mrm{Hyp}(\Gamma)$ a non-zero element $\alpha_H \in V$ such that $\langle \alpha_H \rangle_{\bbC}$ is $W_H$-stable, and choose an element $\alpha_H^\vee \in V^*$ whose kernel is equal to $H$. Then $(\alpha_H,\alpha_H^\vee)$ is a root-coroot pair for any $s \in \mrm{Ref}(\Gamma)$ whose reflection hyperplane is equal to $H$, i.e., for any $s \in W_H \setminus \lbrace 1 \rbrace$, and the expression
\[
(x,y)_H \dopgleich \frac{ \langle x, \alpha_H^\vee \rangle \langle \alpha_H,y \rangle}{\langle \alpha_H,\alpha_H^\vee \rangle} = (x,y)_s
\]
does not depend on the choice of $s \in W_H \setminus \lbrace 1 \rbrace$ and on the choice of $(\alpha_H,\alpha_H^\vee)$.

For the parameter space of the rational Cherednik algebras we define the set $\ol{\boldsymbol{\Omega}}(\Gamma)$ as the set of all pairs $(\Omega,j)$ with $\Omega \in \mrm{Hyp}(\Gamma)/W$ and $0 \leq j \leq e_\Omega-1$. If $(k_{\Omega,j}) \in \bbC^{\ol{\boldsymbol{\Omega}}(\Gamma)}$, then we consider the indices $j$ of $k_{\Omega,j}$ always modulo $e_\Omega$.
\end{paran}

\begin{paran}
The \word{rational Cherednik algebra} $\rH_{t,k} \dopgleich \rH_{t,k}(\Gamma)$ for $\Gamma$ in a parameter $t \in \bbC$ and a parameter family $k \dopgleich (k_{\Omega,j}) \in \bbC^{\ol{\boldsymbol{\Omega}}(\Gamma)}$ is now the quotient of the algebra $\rT(V \oplus V^*) \rtimes \bbC W$, with $\rT(-)$ denoting the tensor algebra, by the ideal generated by the relations
\[
\lbrack x,x' \rbrack = 0 = \lbrack y,y' \rbrack \quad \tn{ for all } x,x' \in V, \; y,y' \in V^*
\]
and
\[
\lbrack x,y \rbrack = t \cdot \langle x,y \rangle + \sum_{H \in \mrm{Hyp}(\Gamma)} (x,y)_H \sum_{s \in W_H \setminus \lbrace 1 \rbrace} \left( \sum_{j=0}^{e_{H}-1} \det(s)^{j} (k_{\Omega_H,j+1}-k_{\Omega_H,j}) \right) s
\]
for $x \in V$ and $y \in V^*$. We will denote the coefficient of $s$ in the above commutator relation by $c_k(s)$.
\end{paran}

\begin{paran}
The first important result about rational Cherednik algebras is the so-called \word{Poincar\'e--Birkhoff--Witt theorem} (\word{PBW theorem} for short) proven in \cite{EG-Symplectic-reflection-algebras} which states that there is a canonical isomorphism
\[
\rS(V) \otimes_{\bbC} \bbC W \otimes_{\bbC} \rS(V^*) \cong \rS(V \oplus V^*) \otimes_{\bbC} \bbC W \cong \rH_{t,k}
\]
of $\bbC$-vector spaces, where $\rS(-)$ denotes the symmetric algebra.
\end{paran}

\subsection{Restricted rational Cherednik algebras}

\begin{paran}
The PBW theorem shows that $\rH_{t,k}$ is an infinite-dimensional $\bbC$-vector space. Brown--Gordon \cite{Brown.K;Gordon.I03Poisson-orders-sympl} have proven that the center of $\rH_{t,k}$ is equal to $\bbC$ whenever $t\neq0$ but that the center of $\rH_{0,k}$ is so large that $\rH_{0,k}$ is even a finite module over it. In fact, Etingof--Ginzburg \cite[Proposition 4.15]{EG-Symplectic-reflection-algebras} (see also Gordon \cite{Gor-Baby-verma}) have proven that the ideal $\fm_k$ of $\rS(V) \otimes_\bbC \rS(V^*)$ generated by
\[
(\rS(V)^W \otimes_{\bbC} \rS(V^*)^W)_+ = \rS(V)^W_+ \otimes_{\bbC} \rS(V^*)^W + \rS(V)^W \otimes_{\bbC} \rS(V^*)_+^W \;,
\]
where $(-)_+$ denotes the ideal generated by the elements of positive degree, maps under the PBW isomorphism into the center of $\rH_{0,k}$. Induced by this isomorphism we thus get an isomorphism
\[
\rS(V)_W \otimes_{\bbC} \bbC W \otimes_{\bbC} \rS(V^*)_W = (\rS(V) \otimes_{\bbC} \bbC W \otimes_{\bbC} \rS(V^*)) /\fm_{k} \cong \rH_{0,k}/\fm_k \rH_{0,k} \gleichdop \ol{\rH}_{k} 
\]
of $\bbC$-vector spaces, where 
\[
\rS(V)_W \dopgleich \rS(V)^W/\rS(V)^W_+ \quad \textnormal{and} \quad \rS(V^*)_W \dopgleich \rS(V^*)^W/\rS(V^*)^W_+
\]
are the \word{coinvariant algebras}. As the coinvariant algebras are of dimension $|W|$, the $\bbC$-algebra $\ol{\rH}_{k}$ is of dimension $|W|^3$. It is called the \word{restricted rational Cherednik algebra} of $\Gamma$ in $k$. %
\end{paran}

\begin{paran}
The following basic representation theoretic properties of $\ol{\rH}_k$ are due to Gordon \cite{Gor-Baby-verma}. First, note that $\ol{\rH}_{k}$ is $\bbZ$-graded by putting $V$ in degree $1$, $\bbC W$ in degree $0$ and $V^*$ in degree $-1$. Then $(\ol{\rH}_{k})_{\leq 0} \cong \rS(V^*) \rtimes \bbC W$ is a graded subalgebra of $\ol{\rH}_{k}$ and as $\bbC W$ is a quotient of this algebra, we can inflate $\bbC W$-modules to $(\ol{\rH}_{k})_{\leq 0}$-modules. For each irreducible character $\lambda \in \Lambda \dopgleich \mrm{Irr}(\bbC W)$ we can now define the corresponding \word{Verma module} $\rM_{k}(\lambda) \dopgleich \ol{\rH}_{k} \otimes_{(\ol{\rH}_{k})_{\leq 0} } \lambda$. We clearly have $\rM_k(\lambda) \cong \rS(V)_W \otimes_\bbC \lambda \cong \bbC W \otimes_\bbC \lambda$ as $\bbC W$-modules so that in particular $\dim_\bbC \rM_k(\lambda) = |W| \cdot \dim_\bbC \lambda$. Due to \cite{Gor-Baby-verma} the Verma module $\rM_k(\lambda)$ is an indecomposable graded $\ol{\rH}_{k}$-module with simple head $\rL_{k}(\lambda)$ and $(\rL_k(\lambda))_{\lambda \in \Lambda}$ is a system of representatives of the isomorphism classes of simple $\ol{\rH}_k$-modules. As any simple $\ol{\rH}_k$-module lies in a unique block, the block structure of $\ol{\rH}_k$ yields a partition of $\Lambda$ whose members are called the \word{Calogero--Moser $k$-families}. We denote this partition by $\mrm{CM}_k$. %
\end{paran}

\begin{paran}
For more details about the following we refer to \cite{Bonnafe.C;Rouquier.R13Cellules-de-Calogero} and \cite{Thiel.UOn-restricted-ration}. The Calogero--Moser families have the property that there exists a non-empty Zariski open subset $\mrm{CMGen}(\Gamma)$ of the parameter space $\bbC^{\ol{\boldsymbol{\Omega}}(\Gamma)}$ on which the Calogero-Moser families have their finest possible form. We will call these families---which are independent of the particular choice of parameters contained in $\mrm{CMGen}(\Gamma)$---the \word{generic Calogero--Moser families} of $\Gamma$. In fact, the generic Calogero--Moser families are precisely the Calogero--Moser $k$-families when considering the parameters $k$ as being algebraically independent. The Calogero--Moser $k$-families for \textit{exceptional} parameters $k$, i.e., for $k$ in the complement $\mrm{CMEx}(\Gamma)$ of $\mrm{CMGen}(\Gamma)$, are then unions of the generic Calogero--Moser families. It is still an open and important problem to understand $\mrm{CMEx}(\Gamma)$ explicitly for the exceptional groups. In \S3 we will present new results in this direction. 

\end{paran}

\subsection{Martino's conjecture}

\begin{paran}
We now come to Martino's conjecture whose historical context we already outlined in the introduction. To state the conjecture in a precise form, we will first quickly recall the definition of Rouquier families. The main references for this are \cite{BroKim02-Familles-de-cara}, \cite{MalRou-Familles-de-caracteres-de-0}, and \cite{Chl09-Blocks-and-famil} but we also want to emphasize the discussion in \cite[chapter 2]{Bonnafe.C;Rouquier.R13Cellules-de-Calogero}, which we will review here. This approach is more general than the one in \cite{Chl09-Blocks-and-famil} but everything works analogously.
\end{paran}

\begin{paran}
Let $\sO$ be the character value ring of $W$, i.e., the subring of $\bbC$ generated by the values of the irreducible characters of $W$, and let $K$ be its quotient field. Let $\biq \dopgleich (\biq_{\Omega,j})_{(\Omega,j) \in \ol{\boldsymbol{\Omega}}(\Gamma)}$ be a family of algebraically independent elements over $\sO$ and let $\sO \lbrack \biq^{\pm 1} \rbrack \dopgleich \sO \lbrack \biq,\biq^{-1} \rbrack$ be the ring of Laurent polynomials in $\biq$ over $\sO$. Let $\sH \dopgleich \sH(\Gamma)$ be the \word{generic Hecke algebra} of $\Gamma$. This is the quotient of $\sO \lbrack \biq^{\pm 1} \rbrack B(\Gamma)$, where $B(\Gamma)$ is the braid group of $\Gamma$, by the ideal generated by the relations
\[
\prod_{j=0}^{e_H-1}\left( \sigma_H - \zeta_{e_H}^j \biq_{\Omega_H,j}^{|\mu(K)|} \right) 
\]
for all $H \in \mrm{Hyp}(\Gamma)$, where $\mu(K)$ is the group of all roots of unity in $K$ and $\sigma_H$ is the generator of the monodromy around $H$ (see \cite{BMR-Complex-Reflection}). Usually, the generic Hecke algebra is considered over the subring $\sO\lbrack \biu^{\pm 1} \rbrack$ or $\bbZ \lbrack \biu^{\pm 1} \rbrack$ of $\sO \lbrack \biq^{\pm 1} \rbrack$, where $\biu \dopgleich (\biu_{\Omega,j})_{(\Omega,j) \in \ol{\boldsymbol{\Omega}}(\Gamma)}$ with $\biu_{\Omega,j} \dopgleich \zeta_{e_\Omega}^j \biq_{\Omega,j}^{|\mu(K)|}$, but the definition here is more convenient for our purposes. 
\end{paran}

\begin{paran}
For the actual work with Hecke algebras one normally wants it to satisfy the following  properties:
\begin{enumerate}[leftmargin=4em,nosep]
\item[(HF)] $\sH$ is a free $\sO\lbrack \biq^{\pm 1} \rbrack$-module of rank $|W|$.
\item[(HS)] There exists a symmetrizing form on $\sH$ satisfying several properties as explained in \cite[4.2.3(1-3)]{Chl09-Blocks-and-famil} (see also \cite[2.60]{Broue.M;Malle.G;Michel.JSplit-Spetses-for-pr}).
\end{enumerate}
These conjectures are known to hold for all Coxeter groups by the classical theory in this case (see \cite[IV, \S2, exercise 23]{Bou-Lie-4-6}). Furthermore, they are known to hold for all complex reflection groups of type $G(m,p,n)$ by \cite[1.17 and the preceding comment]{Broue.M;Malle.G;Michel.J99Towards-spetses.-I}, \cite[theorem 5.1]{Malle.G;Mathas.A98Symmetric-cyclotomic}, and \cite[\S4]{Geck.M;Iancu.L;Malle.G00Weights-of-Markov-tr}. Unfortunately, it seems that there is not much known about these conjectures for non-Coxeter exceptional groups. In \cite[4.7]{Broue.M;Malle.G93Zyklotomische-Heckea} the assumption (HF) is proven for the groups $G_4$, $G_5$, $G_{12}$, and $G_{25}$---but recently Marin has pointed out in \cite{Marin.I12The-freeness-conject} that these proofs might contain a questionable argument. Furthermore, in \cite{Malle.G;Michel.J10Constructing-represe} it is mentioned that (HF) and (HS) have been confirmed by Müller computationally in several cases---but unfortunately this work is not published. It is thus hard to say for which non-Coxeter exceptional groups we can believe that the conjectures hold. Because of this we will, as in \cite{Marin.I12The-freeness-conject}, be careful and consider the cases treated in \cite{Broue.M;Malle.G93Zyklotomische-Heckea} and \cite{Malle.G;Michel.J10Constructing-represe} not as cases where the conjectures are known to hold. Marin has recently proven in \cite{Mar-The-cubic-Hecke-algebra-o-0} however that (HF) holds for the groups $G_4$, $G_{25}$, and $G_{32}$, and there is hope that future work will solve more cases.
\end{paran}

\begin{ass_boxed}
We assume from now on that the properties (HF) and (HS) are satisfied for $\Gamma$.
\end{ass_boxed}

\begin{paran}
It follows from \cite[5.2]{Malle.G99On-the-rationality-a} that the algebra $\ol{\sH} \dopgleich K(\biq) \otimes_{\sO \lbrack \biq^{\pm 1} \rbrack} \sH$ is split. Since this algebra admits a specialization to the (split semi-simple) group algebra $K W$ via $\biq_{\Omega,j} \mapsto 1$, an application of Tits's deformation theorem \cite[7.4.6]{GP-Coxeter-Hecke} shows that $\ol{\sH}$ is also semisimple and that we have induced by this specialization a bijection between the simple $\ol{\sH}$-modules and the simple $KW$-modules.
\end{paran}

\begin{paran}
We write the group ring $\sO\lbrack  \bbC  \rbrack$ of the abelian group $\bbC$ in exponential notation as $\sO \lbrack \biq^{\bbC} \rbrack$ with basis elements $\biq^\alpha$ for $\alpha \in \bbC$ and multiplication $\biq^\alpha \cdot \biq^{\beta} = \biq^{\alpha+\beta}$. %
Since $\bbC$ is a torsion-free abelian group, it can be ordered due to a theorem by Levi (see \cite[theorem 6.31]{Lam-Noncommutative}) and now it is a standard fact (see \cite[theorem 6.29]{Lam-Noncommutative}) that the group ring $\sO \lbrack \biq^{\bbC} \rbrack = \sO \lbrack \bbC \rbrack$ is an integral domain. We denote its quotient field by $K(\biq^\bbC)$. Bonnafé--Rouquier have proven in \cite[chapter 2]{Bonnafe.C;Rouquier.R13Cellules-de-Calogero} that $\sO \lbrack \bbR \rbrack$ is integrally closed but their proof works word-for-word also for $\bbC$ instead of $\bbR$, so $\sO \lbrack \biq^\bbC \rbrack$ is integrally closed.

 If $k \dopgleich (k_{\Omega,j}) \in \ol{\boldsymbol{\Omega}}(\Gamma)$ is a family of elements of $\bbC$, we define the \word{$k$-cyclotomic specialization} as the morphism $\Theta_k^{\mrm{cyc}}: \sO \lbrack \biq^{\pm 1} \rbrack \rarr \sO \lbrack \biq^{\bbC} \rbrack$, $\biq_{\Omega,j} \mapsto \biq^{k_{\Omega,j}}$, and call the algebra $\sH_k^{\mrm{cyc}} \dopgleich (\Theta_k^{\mrm{cyc}})^* \sH = \sO \lbrack \biq^\bbC \rbrack \otimes_{\sO \lbrack \biq^{\pm 1}  \rbrack} \sH$ the \word{$k$-cyclotomic Hecke algebra} of $\Gamma$. The splitting result above and results by Chlouveraki \cite{Chl09-Blocks-and-famil} show that $\ol{\sH}_k^{\mrm{cyc}} \dopgleich K( \biq^{\bbC}) \otimes_{\sO \lbrack \biq^\bbC \rbrack} \sH_k^{\mrm{cyc}}$ is split semi-simple (see \cite[chapter 2]{Bonnafe.C;Rouquier.R13Cellules-de-Calogero} for a precise argument). We thus have induced by specializations bijections between the simple $\ol{\sH}$-modules, the simple $\ol{\sH}_k^{\mrm{cyc}}$-modules and the simple $KW$-modules. 

Now, let $\sO^{\mrm{cyc}}\lbrack \biq^{\bbC} \rbrack$ be the sub $\sO\lbrack \biq^\bbC \rbrack$-algebra of $K(\biq^\bbC)$ generated by the elements $(1-\biq^r)^{-1}$ for $r \in \bbC^\times$. This ring is called the \word{Rouquier ring}. The blocks of the subalgebra $\sO^{\mrm{cyc}}\lbrack \biq^{\bbC} \rbrack \otimes_{\sO\lbrack \biq^\bbC \rbrack} \sH_k^{\mrm{cyc}}$ of $\ol{\sH}_k^{\mrm{cyc}}$ partition the simple $\ol{\sH}_k^{\mrm{cyc}}$-modules---this is a coarser partition than the one induced by the blocks of $\ol{\sH}_k^{\mrm{cyc}}$ itself---and thus the simple $KW$-modules. The families of the simple $KW$-modules obtained in this way are called the \word{Rouquier $k$-families}. 
\end{paran}

\begin{paran}
The parameter dependent behavior of Rouquier families was analyzed by Chlouveraki in \cite{Chl09-Blocks-and-famil} using the notion of essential hyperplanes. The essential hyperplanes are hyperplanes in the parameter space $\bbC^{\ol{\boldsymbol{\Omega}}(\Gamma)}$ of the Rouquier families such that their union $\mrm{RouGen}(\Gamma)$ has a similar property as $\mrm{CMGen}(\Gamma)$, namely away from these hyperplanes the Rouquier families are as fine as possible and independent of the particular parameters. We will call them also the \word{generic Rouquier families}. For parameters in the complement $\mrm{RouEx}(\Gamma)$ of $\mrm{RouGen}(\Gamma)$ some of the generic Rouquier blocks will fuse.%

To review the definition of essential hyperplanes, we first define an \word{essential monomial} of $\Gamma$ to be a monomial $M$ in the variables $\biq_{\Omega,j}$ such that there is a simple $\ol{\sH}$-module whose Schur element contains a factor of the form $\Psi(M)$, where $\Psi$ is a $K$-cyclotomic polynomial with $\Psi(1)$ not invertible in $\sO$. If $M = \prod_{(\Omega,j) \in \ol{\boldsymbol{\Omega}}(\Gamma)} \biq_{\Omega,j}^{n_{\Omega,j}}$ is such a monomial and $k \in \bbC^{\ol{\boldsymbol{\Omega}}(\Gamma)}$, then under the $k$-cyclotomic specialization it specializes to $\biq^{\sum_{(\Omega,j) \in \ol{\boldsymbol{\Omega}}(\Gamma)} n_{\Omega,j} k_{\Omega,j}} \in \sO\lbrack \biq^{\bbC} \rbrack$. The hyperplane in $\bbC^{\ol{\boldsymbol{\Omega}}(\Gamma)}$ defined by $\sum_{(\Omega,j) \in \ol{\boldsymbol{\Omega}}(\Gamma)} n_{\Omega,j} k_{\Omega,j}$ is called the \word{essential hyperplane} associated with $M$. The \word{essential hyperplanes} of $\Gamma$ are now the essential hyperplanes of the essential monomials of $\Gamma$.

Chlouveraki has determined in \cite{Chl09-Blocks-and-famil} all essential hyperplanes and all Rouquier families of all reflection groups. The results for the exceptional groups are contained in the \textsc{CHEVIE} package \cite{Chlouveraki.MGAP-functions-for-th} and can be obtained by using the command \texttt{DisplayAllBlocks(ComplexReflectionGroup(i))}.
\end{paran}

\begin{remark}
Our review of Rouquier families differs slightly from the original discussion in \cite{Chl09-Blocks-and-famil}. Here, only \textit{integral} parameters $k$ satisfying a certain Galois invariance property were allowed (see \cite[4.3.1]{Chl09-Blocks-and-famil}). But this restriction is not necessary at all and the results of \cite{Chl09-Blocks-and-famil} remain the same in the general case as everything is already defined over $\sO \lbrack \biq^{\bbZ} \rbrack$ and as the cyclotomic specializations already split over $K(\biq^{\bbZ})$ by \cite{Chl09-Blocks-and-famil}. We refer again also to \cite[chapter 2]{Bonnafe.C;Rouquier.R13Cellules-de-Calogero}, where this is reviewed for real parameters. 
\end{remark}

\begin{para}
We are finally ready to state Martino's conjecture. %
\end{para}

\begin{conjecture_boxed}[\word{Special parameter conjecture}, \cite{Mar-CM-Rouqier-partition}] \label{conjecture_special}
Let $k \dopgleich (k_{\Omega,j}) \in \bbC^{\ol{\boldsymbol{\Omega}}(\Gamma)}$ and define $k^\sharp \dopgleich (k_{\Omega,j}^\sharp) \in \bbC^{\ol{\boldsymbol{\Omega}}(\Gamma)}$ with $k_{\Omega,j}^\sharp \dopgleich k_{\Omega,-j}$. Then the Calogero--Moser $k$-families are unions of Rouquier $k^\sharp$-families.
\end{conjecture_boxed}
\vspace{6pt}
\begin{conjecture_boxed}[\word{Generic parameter conjecture}, \cite{Mar-CM-Rouqier-partition}] \label{conjecture_generic}
The generic Calogero--Moser families are equal to the generic Rouquier families.
\end{conjecture_boxed}
\vspace{5pt}

\begin{remark}
Our formulation of the conjectures is at first sight slightly different from the original formulation in \cite{Mar-CM-Rouqier-partition} but there is actually no difference as we will argue now.

First, Martino's original formulation of the special parameter conjecture only deals with \textit{integral} parameters---but, as apparent from \cite{Mar-CM-Rouqier-partition} and \cite{Martino.M11Blocks-of-restricted}, this was only due to the assumption on the integrality of the parameters for the Rouquier families in \cite{Chl09-Blocks-and-famil} and we have already argued that this was not necessary.

Second, the term \textit{generic} in Martino's original formulation of the generic parameter conjecture was not defined---it was just stated that for \textit{generic} parameters both families are equal. To be careful, we prove in the following lemma \ref{weak_generic_conjecture} that another more direct and seemingly stronger interpretation is equivalent to our formulation.
\end{remark}

\begin{lemma} \label{conjecture_zariski_stuff}
If there exists a non-empty Zariski open subset $U \subs \bbC^{\ol{\boldsymbol{\Omega}}(\Gamma)}$ such that the Calogero--Moser $k$-families remain the same for all $k \in U$, then these families are the generic Calogero--Moser families and so $U \subs \Phi^{-1}\mrm{CMGen}(\Gamma)$. A similar statement holds for Rouquier families.
\end{lemma}

\begin{proof}
As $\Phi^{-1}\mrm{CMGen}(\Gamma) \subs \bbC^{\ol{\boldsymbol{\Omega}}(\Gamma)}$ is a non-empty Zariski open subset, the intersection $U \cap \Phi^{-1}\mrm{CMGen}(\Gamma)$ is non-empty and so the Calogero--Moser $k$-families for $k \in U$ are the generic Calogero--Moser families. The same argument applies to Rouquier families.
\end{proof}

\begin{lemma} \label{weak_generic_conjecture}
The following assertions are equivalent:
\begin{enum_thm}
\item \label{weak_generic_conjecture:a} The generic Calogero--Moser families are equal to the generic Rouquier families.
\item \label{weak_generic_conjecture:b} There exists a non-empty Zariski open subset $U \subs \bbC^{\ol{\boldsymbol{\Omega}}(\Gamma)}$ such that the Calogero--Moser $k$-families are equal to the Rouquier $k^\sharp$-families for all $k \in U$.
\end{enum_thm}
\end{lemma}

\begin{proof}
First note that the map $\sharp:\bbC^{\ol{\boldsymbol{\Omega}}(\Gamma)} \rarr \bbC^{\ol{\boldsymbol{\Omega}}(\Gamma)}$, $k \mapsto k^\sharp$, is an automorphism of order $2$ of the variety $\bbC^{\ol{\boldsymbol{\Omega}}(\Gamma)}$. As both $\Phi^{-1}\mrm{CMGen}(\Gamma)$ and $\mrm{RouGen}(\Gamma)$ are non-empty open subsets of $\bbC^{\ol{\boldsymbol{\Omega}}(\Gamma)}$, also
\[
X \dopgleich \Phi^{-1}\mrm{CMGen}(\Gamma) \cap (\Phi^{-1}\mrm{CMGen}(\Gamma))^\sharp \cap \mrm{RouGen}(\Gamma) \cap \mrm{RouGen}(\Gamma)^\sharp
\]
is a non-empty open subset of $\bbC^{\ol{\boldsymbol{\Omega}}(\Gamma)}$. This subset is furthermore $\sharp$-stable. 

Now, suppose that \ref{weak_generic_conjecture:a} holds. If $k \in X$, then $k \in \Phi^{-1}\mrm{CMGen}(\Gamma)$ so that the Calogero--Moser $k$-families are equal to the generic Calogero--Moser families. These are in turn by assumption equal to the generic Rouquier families. As $k^\sharp \in X$ and so $k^\sharp \in \mrm{RouGen}(\Gamma)$, the generic Rouquier families are equal to the Rouquier $k^\sharp$-families. Hence, the Calogero--Moser $k$-families are equal to the Rouquier $k^\sharp$-families for all $k \in X$, proving \ref{weak_generic_conjecture:b}.

Conversely, suppose that \ref{weak_generic_conjecture:b} holds. As $U$ is a non-empty open subset, it follows from \ref{conjecture_zariski_stuff} that the Calogero--Moser $k$-families are equal to the generic Calogero--Moser families for all $k \in U$. In the same way, using that $U^\sharp$ is also a non-empty open subset, the Rouquier $k^\sharp$-families are equal to the generic Rouquier families. This already shows \ref{weak_generic_conjecture:a}.
\end{proof}

\begin{para}
Although the generic parameter conjecture does not involve the explicit parameter correspondence $k \leftrightarrow k^\sharp$, it still implies that the special parameter conjecture holds for Zariski almost all parameters as the following lemma shows.
\end{para}

\begin{lemma}
If the generic Calogero--Moser families are unions of generic Rouquier families, then Martino's special parameter conjecture holds for all $k \in \Phi^{-1}\mrm{CMGen}(\Gamma)$.
\end{lemma}

\begin{proof}
Suppose that $k  \in \Phi^{-1}\mrm{CMGen}(\Gamma)$. Then the Calogero--Moser $k$-families are equal to the generic Calogero--Moser $k$-families, which are by assumption unions of generic Rouquier families and these are in turn unions of Rouquier $k^\sharp$-families. Hence, the Calogero--Moser $k$-families are unions of Rouquier $k^\sharp$-families and so
Martino's special parameter conjecture holds for $k$.%
\end{proof}

\begin{para}
The following theorem summarizes the current status of the conjecture.
\end{para}

\begin{thm}[\cite{GorMar-Calogero-Moser-space-rest-0}, \cite{Mar-CM-Rouqier-partition}, \cite{Bel-The-Calogero-Moser-partit-0}, \cite{Bel-Singular-CM}, \cite{Martino.M11Blocks-of-restricted}]
Martino's conjecture holds in its full form for the whole infinite series $G(m,p,n)$, and the generic parameter conjecture also holds for the exceptional group $G_4$. \qed
\end{thm}

\begin{para}
The aim of this article is to investigate Martino's conjecture for further exceptional groups. As the exceptional groups unfortunately do not have a coherent description as the groups $G(m,p,n)$, results in this direction will (so far) be less structural and more computational. However, the key ingredients allowing us to derive some results are of theoretical nature and we will now discuss them. 
\end{para}

\subsection{Dimensions of simple modules and Calogero--Moser families}

\begin{parani}
We will now review an important connection between the dimensions of simple $\ol{\rH}_k$-modules and Calogero--Moser $k$-families. %
\end{parani}

\begin{thm}[{\cite[1.7]{EG-Symplectic-reflection-algebras}}] \label{simple_module_dimension_bound}
For all $\lambda$ and all $k$ the estimate $
\dim_{\bbC} \rL_k(\lambda) \leq |W|$ holds. \qed
\end{thm}

\begin{para}
Although the following result is well-known and follows in conjunction with the above theorem from a characterization of the Azumaya points of the center of $\rH_k$ due to Brown (discussed by Gordon \cite[7.2]{Gor-Baby-verma}), we give here a simple isolated proof of this fact.
\end{para}

\begin{prop} \label{dim_less_non_trivial}
If $\dim_{\bbC} \rL_k(\lambda) < |W|$, then $\lambda$ lies in a non-singleton Calogero--Moser $k$-family.
\end{prop}

\begin{proof}
Suppose that $L \dopgleich \rL_k(\lambda)$ lies in a singleton Calogero--Moser $k$-family. Since $M \dopgleich \rM_k(\lambda)$ is indecomposable, we thus have $\lbrack M \rbrack = a \lbrack L \rbrack$ in the Grothendieck group $\rG_0(\ol{\rH}_k)$ for some $a \in \bbN_{>0}$. This relation implies that 
\[
a \dim_\bbC L = \dim_\bbC M  = |W| \cdot \dim_\bbC \lambda
\] 
and since we assumed $\dim_\bbC L < |W|$, we conclude that $a> \dim_\bbC \lambda$. 
Let $(-)_W \dopgleich \mrm{res}_{\bbC W}^{\ol{\rH}_k} : \rG_0(\ol{\rH}_k) \rarr \rG_0(\bbC W)$ be the canonical morphism. Since $M \cong \bbC W \otimes_\bbC \lambda$ as $\bbC W$-modules, we get
\[
a \lbrack L \rbrack_W = \lbrack M \rbrack_W = \lbrack \bbC W \otimes_\bbC \lambda \rbrack_W = \dim_\bbC \lambda \cdot \lbrack 1_W \rbrack + \sum_{ \mu \in \Lambda \setminus \lbrace 1_W \rbrace} a_\mu \lbrack \mu \rbrack \;
\]
for some $a_\mu \in \bbN$. In the above we used the fact that $\lbrack \bbC W \otimes_{\bbC} \lambda, 1_W \rbrack = \dim_{\bbC} \lambda$, which follows at once from the relation $(\chi \otimes \psi,\varphi) = (\chi,\psi^* \otimes \varphi)$ for characters $\chi,\varphi,\psi$ of $W$. But the above relation implies that $\dim_\bbC \lambda$ must be divisible by $a$, contradicting the estimate $a> \dim_\bbC \lambda$ derived above. Hence, $L$ lies in a non-singleton Calogero--Moser $k$-family. 
\end{proof}

\subsection{Supersingular characters} \label{supersingular_section}

\begin{parani}
 \label{fake_degrees}
Recall that $\rM_k(\lambda)$ is $\bbZ$-graded and therefore its head $\rL_k(\lambda)$ is also $\bbZ$-graded. As $\rM_k(\lambda)$ is by definition concentrated in degree $\bbZ_{\geq 0}$, the same applies to $\rL_k(\lambda)$. In the extreme case that $\dim_\bbC \rL_k(\lambda) = |W|$ there is a general formula for the Poincar\'e series of $\rL_k(\lambda)$ due to Bellamy.

To present this formula, first recall (see also  \cite[\S24]{Kan-Reflection}) that the coinvariant algebra $\rS(V)_W$ is as a $\bbC W$-module isomorphic to $\bbC W$ and is therefore a $\bbZ_{\geq 0}$-graded version of the natural $\bbC W$-module. Each $\lambda \in \Lambda$ thus has a graded multiplicity in $\rS(V)_W$, given by the expression
\[
f_\lambda(t) \dopgleich \sum_{i \in \bbZ} \lbrack (\rS(V)_W)_i : \lambda \rbrack \cdot t^i  \in \bbZ \lbrack t \rbrack = \left( \prod_{i=1}^n 1-t^{d_i} \right) \frac{1}{|W|} \sum_{w \in W} \frac{\lambda(w)}{\mrm{det} (1-wt)} \;,
\]
where $(\rS(V)_W)_i$ denotes the homogeneous component of degree $i$ of $\rS(V)_W$ and $(d_i)_{i=1}^n$ are the invariant degrees of $\Gamma$. 
The Poincar\'e series of $\rS(V)_W$ has the explicit form
\[
\rP_{\rS(V)_W}(t) = \sum_{i \in \bbZ} \dim_{\bbC} (\rS(V)_W)_i \cdot t^i = \prod_{i=1}^{n} \frac{1-t^{d_i}}{1-t} = \prod_{i=1}^n \sum_{j=0}^{d_i-1} t^j \;.
\]
\end{parani}

\begin{thm}[{\cite[3.3]{Bel-Singular-CM}}] \label{p_series_formula}
If $\dim_{\bbC} \rL_k(\lambda) = |W|$, then the Poincar\'e series of $\rL_k(\lambda)$ is equal to
\[
\rP_{\rL_k(\lambda)}(t) = \frac{\dim_\bbC(\lambda) t^{b_\lambda} \rP_{\rS(V)_W}(t)}{f_\lambda(t)} \in \bbZ \lbrack t \rbrack \;,
\] 
where $b_\lambda$ is the \word{trailing degree} of $f_\lambda(t)$, i.e., the order of zero of $f_\lambda(t)$ at $t=0$.%
\end{thm}

\begin{defn}
If $f_\lambda(t)$ does not divide $\dim_\bbC(\lambda) t^{b_\lambda} \rP_{\rS(V)_W}(t)$, we say that $\lambda$ is \word{supersingular}. Note that this independent of the parameter $k$.
\end{defn}

\begin{prop} \label{supersingular_nonsingleton}
If $\lambda$ is supersingular, then $\dim \rL_k(\lambda) < |W|$ and $\rL_k(\lambda)$ lies in a non-singleton Calogero--Moser $k$-family for all $k$.
\end{prop}

\begin{proof}
By theorem \ref{simple_module_dimension_bound} we know that $\dim_\bbC \rL_k(\lambda) \leq |W|$. The supersingularity of $\lambda$ thus implies by theorem \ref{p_series_formula} that $\dim_\bbC \rL_k(\lambda) < |W|$ and then we deduce from proposition \ref{dim_less_non_trivial} that $\rL_k(\lambda)$ lies in a non-singleton Calogero--Moser $k$-family.
\end{proof}

\begin{para}
Due to the explicit formulas given above the notion of supersingularity yields a very effective method for deducing information about the simple modules and the blocks of $\ol{\rH}_k$. We emphasize that this method was already used by Bellamy \cite{Bel-Singular-CM} to deduce that for any exceptional group $G_i$ with $i>4$ there exists a supersingular irreducible character and so the Calogero--Moser families are always non-trivial for these groups. This was a very important result but unfortunately it does not yield any precise information about the actual structure of the Calogero--Moser families. We will now combine this result by Bellamy with a further simple theoretical idea which will finally allow us to actually determine the precise structure of Calogero--Moser families for generic parameters for some exceptional groups. 
\end{para}

\subsection{Euler families} \label{euler_section}

\begin{paran} \label{euler_element_definition}
Our second theoretical ingredient is a certain non-trivial central element, the so-called Euler element, in $\rH_{0,k}$ for any parameter $k$. This element was already used in \cite{Dunkl.C;Opdam.E03Dunkl-operators-for-}, \cite{GinGuaOpd-On-the-category-scr-O-for-0}, \cite{Gordon.I08Quiver-varieties-cat}, and \cite{Bonnafe.C;Rouquier.R13Cellules-de-Calogero}. It thus has a long history but this article employs it for the first time also for the exceptional complex reflection groups. Before we recall its definition we note that if $(x_i)_{i=1}^n$ is a basis of $V$ and $(y_i)_{i=1}^n$ is its dual basis, then the element $\sum_{i=1}^n x_i y_i \in \rT(V \oplus V^*)$ is independent of the choice of the basis. To see this, suppose that $(x_i')_{i=1}^n$ is another basis of $V$ with dual basis $(y_i')_{i=1}^n$. Let $g$ be the automorphism of $V$ mapping $x_i$ to $x_i'$ for all $i$. Since $(^gy_i)(x_j') = {(^gy_i)}(^gx_j) = y_i(^{g^{-1}g}x_j) = y_i(x_j) = \delta_{ij}$, it follows that $y_i' = {^gy_i}$ for all $i$. Let $A \dopgleich (a_{ij})$ be the matrix of $g$ acting on $V$ in the basis $(x_i)_{i=1}^n$ and let $B \dopgleich (b_{ij})$ be the matrix of $g$ acting on $V^*$ in the basis $(y_i)_{i=1}^n$. Then $B = (A^{\rt})^{-1}$ and consequently
\begin{align*}
\sum_{i=1}^n x_i' y_i' &= \sum_{i=1}^n {^gx_i} {^gy_i} = \sum_{i=1}^n \left( \sum_{j=1}^n a_{ji} x_j \right) \left( \sum_{k=1}^n b_{ki} y_k \right) = \sum_{i,j,k=1}^n a_{ji} b_{ki} x_j y_k \\ &= \sum_{j,k=1}^n \left( \sum_{i=1}^n a_{ji} b_{ki} \right) x_j y_k = \sum_{j,k=1}^n (AB^{\rt})_{j,k} x_jy_k = \sum_{j,k=1}^n \delta_{jk} x_j y_k = \sum_{i=1}^n x_iy_i \;.
\end{align*}
We can now unambiguously define the \word{Euler element} as 
\[
\mrm{eu}_k \dopgleich \sum_{i=1}^n x_iy_i - \sum_{s \in \mrm{Ref}(\Gamma)} \frac{1}{1-\eps_s} c_k(s) (s - \eps_s) \in \rH_{0,k} \;,
\] 
where $(x_i)_{i=1}^n$ is a basis of $V$ and $(y_i)_{i=1}^n$ is its dual basis. Recall that $\eps_s$ was defined as the non-trivial eigenvalue of $s$. The following lemma is well-known but we still give a proof here for completeness.
\end{paran}

\begin{lemma}
The Euler element $\mrm{eu}_k$ is a non-zero central element in $\rH_{0,k}$.
\end{lemma}

\begin{proof}
Clearly, $\mrm{eu}_k$ is non-zero as its group algebra part 
\[
\check{\mrm{eu}}_k \dopgleich -\sum_{s \in \mrm{Ref}(\Gamma)} \frac{1}{1-\eps_s} c_k(s) (s - \eps_s) \in \bbC W
\]
is non-zero. To show that $\mrm{eu}_c$ is central, it suffices to prove that it commutes with elements from $V$, $V^*$, and $W$ since $\rH_{0,k}$ is generated as a  $\bbC$-algebra by such elements. Let $(x_i)_{i=1}^n$ be a basis of $V$ and let $(y_i)_{i=1}^n$ be its dual basis. If $w \in W$, then clearly $w \check{\mrm{eu}}_k = \check{\mrm{eu}}_k w$ and 
\[
w \left( \sum_{i=1}^n x_i y_i \right) = \sum_{i=1}^n wx_i y_i = \sum_{i=1}^n {^w x_i w} y_i = \sum_{i=1}^n {^wx_i} {^w y_i} w = \left( \sum_{i=1}^n x_i y_i \right)w
\]
as seen in \ref{euler_element_definition}. Hence, $\mrm{eu}_k$ commutes with $w$. 

Now, let $x \in V$. Then
\[
\lbrack \mrm{eu}_k,x \rbrack = - \sum_{i=1}^n x_i \lbrack x, y_i \rbrack + \lbrack \check{\mrm{eu}}_k, x \rbrack \;.
\]
We have
\begin{align*}
\lbrack \check{\mrm{eu}}_k, x \rbrack & = \check{\mrm{eu}}_k x - x \check{\mrm{eu}}_k = -\sum_{s \in \mrm{Ref}(\Gamma)} \frac{1}{1-\eps_s} c_k(s)(s-\eps_s) x + \sum_{s \in \mrm{Ref}(\Gamma)} \frac{1}{1-\eps_s} c_k(s) x (s-\eps_s) \\
&= \sum_{s \in \mrm{Ref}(\Gamma)} \frac{1}{1-\eps_s} c_k(s) ( xs - sx) = \sum_{s \in \mrm{Ref}(\Gamma)} \frac{1}{1-\eps_s} c_k(s)(x - {^sx}) s 
\end{align*}
and therefore
\begin{align*}
\lbrack \mrm{eu}_k,x \rbrack &= - \sum_{i=1}^n x_i \left(\sum_{s \in \mrm{Ref}(\Gamma)} \frac{\langle x,\alpha_s^\vee \rangle \langle \alpha_s,y_i \rangle}{\langle \alpha_s,\alpha_s^\vee \rangle} c_k(s) s\right) + \sum_{s \in \mrm{Ref}(\Gamma)} \frac{1}{1-\eps_s} c_k(s)(x - {^sx}) s \\
&= \sum_{s \in \mrm{Ref}(\Gamma)}  \left( \frac{1}{1-\eps_s} (x-{^sx}) - \frac{\langle x, \alpha_s^\vee \rangle}{\langle \alpha_s,\alpha_s^\vee \rangle} \sum_{i=1}^n x_i\langle \alpha_s,y_i \rangle \right) \;.
\end{align*}
The summand for $s \in \mrm{Ref}(\Gamma)$ in the above expression is zero whenever $x$ is fixed by $s$, since then $x \in \Ker(\alpha_s^\vee)$ and $x-{^sx} = 0$. If $x$ is in the complement $\langle \alpha_s \rangle_{\bbC}$ of $\Ker(\alpha_s^\vee)$, then $^sx = \eps_s x$, and the summand is in this case also easily seen to be zero. In total, we have proven that $\lbrack \mrm{eu}_k,x \rbrack = 0$. In the same way we can prove that $\lbrack \mrm{eu}_k,y \rbrack = 0$ for all $y \in V^*$ and so we have proven that $\mrm{eu}_k$ is central in $\rH_{0,k}$.
\end{proof}

\begin{paran}
The image of the Euler element $\mrm{eu}_k$ in $\ol{\rH}_k$, which we again denote by $\mrm{eu}_k$, is also central, and it is also non-zero since the group algebra part of $\mrm{eu}_k$ is not contained in $\fm_k \rH_{0,k}$. This is a very important aspect, since if $\omega_{\rL_k(\lambda)} : \rZ(\ol{\rH}_k) \rarr \bbC$ denotes the central character of the simple $\ol{\rH}_k$-module $\rL_k(\lambda)$, then the fibers of the map $\Lambda \rarr \bbC$, $\lambda \mapsto \omega_{\rL_k(\lambda)}(\mrm{eu}_k)$, yield a partition of $\Lambda$ which we denote by $\mrm{Eu}_k$ and whose members we call the \word{Euler $k$-families}. As the partition of the simple $\ol{\rH}_k$-modules defined by the block structure of $\ol{\rH}_k$ is determined by the values of the central characters on \textit{all} central elements of $\ol{\rH}_k$, the partition $\mrm{Eu}_k$ is coarser than $\mrm{CM}_k$. Although this is a simple observation, it still yields a very useful access point for determining $\mrm{CM}_k$ in some situations because the Euler families are very easy to compute as we will see.
\end{paran}

\begin{prop}
The following holds:
\begin{enum_thm}
\item The group algebra part 
\[
\check{\mrm{eu}}_k \dopgleich -\sum_{s \in \mrm{Ref}(\Gamma)} \frac{1}{1-\eps_s} c_k(s) (s - \eps_s) \in \bbC W
\]
of $\mrm{eu}_k$ is a non-zero central element in $\bbC W$.
\item For $\lambda \in \Lambda$ the Euler element $\mrm{eu}_k$ acts on $\rM_k(\lambda)$, and thus on its head $\rL_k(\lambda)$, by the scalar $\omega_\lambda(\check{\mrm{eu}}_k)$, where $\omega_\lambda: \rZ(\bbC W) \rarr \bbC$ is the central character of $\lambda$. This value is given by
\[
\omega_{\rL_k(\lambda)}(\check{\mrm{eu}}_k) = \omega_\lambda(\check{\mrm{eu}}_k) = \frac{1}{\lambda(1)} \sum_{s \in \mrm{Ref}(\Gamma)} \frac{c_k(s)}{1-\eps_s}(\eps_s \lambda(1) - \lambda(s) )  \;.
\]
\end{enum_thm}
\end{prop}

\begin{proof}
The first assertion is easy to see. To prove the second, recall that $\rM_k(\lambda) \cong \rS(V)_W \otimes_{\bbC} \lambda$ as $\bbC W$-modules. If $f \otimes u \in \rM_k(\lambda)$, then, as $\mrm{eu}_k$ is central, we have
\begin{align*}
\mrm{eu}_k (f \otimes u) & = \mrm{eu}_k f \otimes u = f \mrm{eu}_k \otimes u \\ \displaybreak[0]
& = f \left(\sum_{i=1}^n x_i y_i - \sum_{s \in \mrm{Ref}(\Gamma)} \frac{1}{1-\eps_s} c_k(s)(s-\eps_s) \right) \otimes u \\ \displaybreak[0]
& = \sum_{i=1}^n f x_i y_i \otimes u - f \left( \sum_{s \in \mrm{Ref}(\Gamma)} \frac{1}{1-\eps_s} c_k(s)(s-\eps_s) \right) \otimes u \\ \displaybreak[0]
&= 0 - f \otimes \sum_{s \in \mrm{Ref}(\Gamma)} \frac{1}{1-\eps_s} c_k(s)(s-\eps_s)u \\ \displaybreak[0]
& = f \otimes \omega_{\lambda}(\check{\mrm{eu}}_k) u = \omega_{\lambda}(\check{\mrm{eu}}_k)  (f \otimes u) \;,
\end{align*}
where we used the fact that $V^*$ acts as zero on $\rM_k(\lambda)$ by definition. This already shows that $\mrm{eu}_k$ acts by $\omega_{\lambda}(\check{\mrm{eu}}_k) $ on $\rM_k(\lambda)$.  
\end{proof}

\begin{cor} \label{in_same_euler_family}
Two irreducible characters $\lambda,\mu \in \Lambda$ lie in the same Euler $k$-family if and only if 
\[
p_{\lambda,\mu}(k) \dopgleich \sum_{s \in \mrm{Ref}(\Gamma)} \frac{c_k(s)}{1-\eps_s} \left( \frac{\lambda(s)}{\lambda(1)} - \frac{\mu(s)}{\mu(1)} \right) = 0 \;.\vspace{-\baselineskip}
\]  \qed
\end{cor}

\begin{paran} \label{euler_variety}
Similar to the Calogero--Moser families the Euler families have the property that there exists a non-empty Zariski open subset $\mrm{EuGen}(\Gamma) \subs \bbC^{\ol{\boldsymbol{\Omega}}(\Gamma)}$ on which the Euler families have their finest possible form---we will call them the \word{generic Euler families}---and that for parameters contained in the complement $\mrm{EuEx}(\Gamma)$ of $\mrm{EuGen}(\Gamma)$ the Euler families are unions of generic Euler families. We call $\mrm{EuEx}(\Gamma)$ the \word{Euler variety} of $\Gamma$. Due to the description of the Euler families above we see that the Euler variety is simply the zero locus of those $p_{\lambda,\mu}$ which are not already constantly zero. As the $p_{\lambda,\mu}$ are homogeneous polynomials of degree one when considering the family $k$ as indeterminates, the Euler variety is a union of hyperplanes. We also see that two irreducible characters $\lambda,\mu \in \Lambda$ lie in a common generic Euler family if and only if
\[
\mu(1) \lambda(s) = \lambda(1) \mu(s)
\]
for all $s \in \mrm{Ref}(\Gamma)$. We can thus easily compute the generic Euler families from the character table of $W$ and this is what makes them a nice tool for our purposes.
\end{paran}

\begin{para}
The concept of Euler families allows us to deduce the following general result.
\end{para}

\begin{prop}
If $\lambda \in \Lambda$ is one-dimensional, it lies in a singleton Calogero--Moser $k$-family for every $k \in \mrm{CMGen}(\Gamma) \cup \mrm{EuGen}(\Gamma)$. 
\end{prop}

\begin{proof}
If we can show the assertion for all $k \in \mrm{EuGen}(\Gamma)$, then it also holds for all $k \in \mrm{CMGen}(\Gamma)$ by lemma \ref{conjecture_zariski_stuff}. It furthermore suffices to show that each one-dimensional character lies in a singleton generic Euler family as the Calogero--Moser families are refinements of the Euler families.
So, suppose that $\mu \in \Lambda$ lies in the same generic Euler family as $\lambda$. By corollary \ref{in_same_euler_family} this means that $\mu(1)\lambda(s) = \lambda(1)\mu(s) = \mu(s)$ for all $s \in \mrm{Ref}(\Gamma)$. Since $\lambda$ is one-dimensional, $\lambda:W \rarr \bbC^\times$ is actually a group morphism so that for any $w \in W$ we have $\lambda(w)^{\mrm{Ord}(w)} = \lambda(w^{\mrm{Ord(w)}}) = \lambda(1) = 1$, i.e., $\lambda(w)$ is a root of unity of order dividing that of $w$. Hence, $|\mu(s)| = |\mu(1) \lambda(s)| = \mu(1) \gleichdop n$. But according to \cite[2.27a]{Isa-Character-theory} this means that $s$ is contained in the center of the character $\mu$ and therefore $\rho_\mu(s) = \eps_s \rI_n$ for some $\eps_s \in \bbC$, where $\rho_\mu: W \rarr \GL_n(\bbC)$ denotes the representation corresponding to $\mu$ and $\rI_n$ is the identity matrix. As $W$ is generated by the $s \in \mrm{Ref}(\Gamma)$, it follows that $\rho_\mu(w)$ is a multiple of $\rI_n$ for all $w \in W$, and now $\rho_\mu$ already has to be one-dimensional as it is irreducible. The above equation now becomes $\mu(s) = \lambda(s)$ for all $s \in \mrm{Ref}(\Gamma)$ and this implies $\mu = \lambda$ as both are one-dimensional. 
\end{proof}

\begin{remark}
The above result was also proven at about the same time by Bonnaf\'e--Rouquier \cite[9.5.10]{Bonnafe.C;Rouquier.R13Cellules-de-Calogero}.%
\end{remark}

\begin{para}
Now we come to our key argument to be used in \S3. For a closer analysis we first introduce the following notion.
\end{para}

\begin{defn}
We say that an Euler $k$-family $\sF$ is \word{good} if it is of one of the following types:
\begin{enum_thm}
\item $|\sF| = 1$.
\item $|\sF| = 2$ and at least one character in $\sF$ is supersingular.
\item $|\sF| = 3$ and all characters in $\sF$ are supersingular. 
\end{enum_thm}
Otherwise, we say that $\sF$ is \word{bad}. 
\end{defn}

\begin{prop} \label{trick}
Every good Euler $k$-family is already a Calogero--Moser $k$-family.
\end{prop}

\begin{proof}
Let $\sF$ be a good Euler $k$-family. If $|\sF| = 1$, then this must already be a Calogero--Moser family since the Calogero--Moser families are finer than the Euler families. If $\sF$ is of the other two types, suppose it is not a Calogero--Moser family. Again, as the Calogero--Moser families are finer than the Euler families, the family $\sF$ would split into several Calogero--Moser families. But now the assumptions on $\sF$ imply that there would be a supersingular character in $\sF$ which would form a singleton Calogero--Moser family. This is, however, not possible by proposition \ref{supersingular_nonsingleton} and so $\sF$ must be already be a Calogero--Moser family.
\end{proof}

\begin{para}
For larger Euler families we cannot decide by such simple methods if they are already Calogero--Moser families or if they split into several Calogero--Moser families. In some cases, however, this is already enough for generic parameters as we will see in the next paragraph.
\end{para}

\section{Results}

\begin{paran}
Due to the explicit formulas given in \S\ref{supersingular_section} and \S\ref{euler_section} we can easily compute the bad generic Euler families for each exceptional group---all we need is just the character table and the invariant degrees of the group (see \ref{shephard_todd}). The following table \ref{bad_generic_euler_families} summarizes the results. %
\begin{table}[H] 
\footnotesize
\centering
\begin{minipage}[t]{0.215\linewidth}
\centering
\begin{tabular}[t]{|l|p{4em}|}
\hline
Group & Bad fam. \\
\hline \hline
$G_4$ &--\\\hline
$G_5$ &--\\\hline
$G_6$ &--\\\hline
$G_7$ & $3^2$ \\\hline
$G_8$ &--\\\hline
$G_9$ & $4^1$ \\\hline
$G_{10}$ &--\\\hline
$G_{11}$ & $3^8, 4^3$ \\\hline
$G_{12}$ & $4^1$ \\\hline
$G_{13}$ & $6^1$ \\\hline
$G_{14}$ & $3^3$ \\\hline
$G_{15}$ & $4^3$ \\\hline
\end{tabular}
\end{minipage}
\begin{minipage}[t]{0.3\linewidth}
\centering
\begin{tabular}[t]{|l|p{5.8em}|}
\hline
Group & Bad fam. \\
\hline \hline
$G_{16}$ & $5^1$ \\\hline
$G_{17}$ & $4^5, 5^2$ \\\hline
$G_{18}$ & $3^5, 5^3$ \\\hline
$G_{19}$ & $3^{20}, 4^{15}, 5^6, 6^5$ \\\hline
$G_{20}$ & $9^1$ \\\hline 
$G_{21}$ & $4^6, 6^3$ \\\hline
$G_{22}$ & $10^1$ \\\hline
$G_{23} = H_3$ &--\\\hline
$G_{24}$ &--\\\hline
$G_{25}$ &--\\\hline
$G_{26}$ &--\\\hline
$G_{27}$ & $3^2, 4^1, 6^2$ \\\hline
\end{tabular}
\end{minipage}
\begin{minipage}[t]{0.4\linewidth}
\begin{tabular}[t]{|l|p{11.2em}|}
\hline
Group & Bad fam. \\
\hline \hline
$G_{28} = F_4$ & $5^1$ \\\hline
$G_{29}$ & $4^2, 9^1$ \\\hline
$G_{30} = H_4$ & $16^1$ \\\hline
$G_{31}$ & $2^2$, $4^2$, $6^2$, $13^1$ \\\hline
$G_{32}$ & $4^3$ \\\hline
$G_{33}$ & $4^2, 7^2$ \\\hline
$G_{34}$ & $3^2, 4^2, 5^2, 6^2, 7^6, 9^3, 17^2, 19^1$ \\ \hline
$G_{35} = E_6$ & $2^2, 3^2, 5^1$ \\ \hline
$G_{36} = E_7$ & $3^8, 5^4$ \\ \hline
$G_{37} = E_8$ & $2^2, 3^{10}, 4^2, 5^2, 6^4, 18^1$ \\ \hline
\end{tabular}
\end{minipage} \caption{Bad generic Euler families for the exceptional groups.}  \label{bad_generic_euler_families}
\end{table}

\noindent Using \ref{trick}, \ref{conjecture_zariski_stuff}, and \ref{euler_variety} we immediately deduce from this table the following theorem.
\end{paran}

\begin{thm_boxed} \label{bad_generic_euler_families_thm}
For precisely the exceptional groups 
\[
W \in \lbrace G_4,G_5,G_6,G_8,G_{10}, G_{23}=H_3,G_{24},G_{25},G_{26} \rbrace
\]
there are no bad generic Euler families and so
\[
\mrm{Eu}_k = \mrm{CM}_k \quad \tn{for all} \ k \notin \Phi^{-1}\mrm{EuEx}(W) \;.
\] 
In particular, the generic Euler families are equal to the generic Calogero--Moser families and $\mrm{CMEx}(W) \subs \mrm{EuEx}(W)$. \qed
\end{thm_boxed}
 
 \begin{para}
We note that except for the group $G_5$ and $G_{10}$ all groups in \ref{bad_generic_euler_families_thm} are spetsial (see \cite[\S8]{Mal00-On-the-generic-d}).  
\end{para}

\begin{para}
A closer look at the generic Calogero--Moser families of the above groups and a comparison to the generic Rouquier families in the following sections \S\ref{g4} to \S\ref{g26} yields a proof of the following theorem.
 \end{para}
 
 \begin{thm_boxed} \label{main_theorem_2}
 For all the groups 
 \[
 W \in \lbrace G_4,G_5,G_6,G_8,G_{10}, G_{23}=H_3,G_{24},G_{25},G_{26} \rbrace
 \]
 the following holds:
 \begin{enum_thm}
 \item Martino's special parameter conjecture holds for all parameters $k \in \Phi^{-1}\mrm{CMGen}(W) \sups \Phi^{-1}\mrm{EuGen}(W)$.
 \item Martino's generic parameter conjecture holds---except for the group $G_{25}$ where it fails.
 \item $\Phi^{-1}\mrm{EuEx}(W)$ is $\sharp$-stable and 
 \[
\mrm{RouEx}(W) \subs \Phi^{-1}\mrm{EuEx}(W) \sups \Phi^{-1}\mrm{CMEx}(W) \;.
\]
\qed
 \end{enum_thm}
 \end{thm_boxed}

 \begin{paran}

Each of the upcoming sections is devoted to one particular of the above groups---starting with $G_4$ and $G_{25}$ because here something interesting happens. %
As there are no canonical explicit realizations of the exceptional groups, we usually---if nothing else is mentioned---include the realization used in \cite{CHEVIE-JM-4} and we will then also use the labeling of the irreducible characters of the reflection groups as in \cite{CHEVIE-JM-4} because a consistent labeling is crucial for the comparison with the results in \cite{Chlouveraki.MGAP-functions-for-th}. Mostly, an irreducible character $\chi$ can uniquely be labeled by $\phi_{d,b}$, where $d \dopgleich \chi(1)$ is the degree of $\chi$ and $b$ is the trailing degree of the fake degree $f_\chi(t)$ of $\chi$. Unfortunately, this \word{$(d,b)$-pair} does not always yield a unique label. According to more ore less fixed rules those characters with the same $(d,b)$-pair are labeled in \cite{CHEVIE-JM-4} with additional primes attached---and this is where the explicit realizations of the groups will be important. The labelings used in \cite{CHEVIE-JM-4} can be obtained using the command \texttt{CharNames(CharTable(W))}, where $W$ is an exceptional complex reflection group.

To obtain explicit knowledge about generic parameters we will also always provide an explicit presentation of the Euler variety. To simplify notations we have decomposed the Euler variety $\Phi^{-1} \mrm{EuEx}(\Gamma) \subs \bbC^{\ol{\boldsymbol{\Omega}}(\Gamma)}$ into orbits under the action of Young subgroups of the symmetric group $\rS_{|\ol{\boldsymbol{\Omega}}(\Gamma)|}$ on the projective space of hyperplanes
\[
\mrm{Gr}(\bbC^{\ol{\boldsymbol{\Omega}}(\Gamma)}, |\ol{\boldsymbol{\Omega}}(\Gamma)|-1) \cong \bbP(\bbC^{\ol{\boldsymbol{\Omega}}(\Gamma)}) \cong \bbC^{\ol{\boldsymbol{\Omega}}(\Gamma)}/\bbC^\times
\]
by coordinate permutations. So, for example, if $\Gamma$ has two conjugacy classes of reflections of order $2$, and $\Omega_1$ and $\Omega_2$ denote the corresponding reflection hyperplane orbits, then the parameters in $\bbC^{\ol{\boldsymbol{\Omega}}(\Gamma)}$ are of the form $k \dopgleich (k_{1,0}, k_{1,1}, k_{2,0},k_{2,1})$ and the notation
\[
\Sigma_{\lbrace \lbrace 1,2 \rbrace, \lbrace 3,4 \rbrace \rbrace}.(1,2,3,4)
\]
denotes the orbit of the hyperplane $k_{1,0} + 2k_{1,1} + 3k_{2,0} + 4k_{2,1}$ under the Young subgroup  $\Sigma_{\lbrace \lbrace 1,2 \rbrace, \lbrace 3,4 \rbrace \rbrace}$ of $\rS_4$. The hyperplanes in this orbit are
\[
\begin{array}{l}
k_{1,0} + 2k_{1,1} + 3k_{2,0} + 4k_{2,1} \\
2k_{1,0} + k_{1,1} + 3k_{2,0} + 4k_{2,1} \\
k_{1,0} + 2 k_{1,1} + 4k_{2,0} + 3k_{2,1} \\
2k_{1,0} + k_{1,1} + 4k_{2,0} + 3k_{2,1} \;.
\end{array}
\]
\end{paran} %

\vspace{-10pt}
\subsection{$G_4$} \label{g4}

\begin{paran}
The group $G_4$ is the smallest exceptional reflection group and is isomorphic to $\SL_2(\bbF_3)$, so its order is equal to 24. It can be realized as the matrix group over $\bbQ(\zeta)$, where $\zeta \dopgleich \zeta_3$, generated by the reflections
\[
s \dopgleich \begin{pmatrix} 1 & 0 \\ 0 & \zeta \end{pmatrix}, \quad t \dopgleich \frac{1}{3} \begin{pmatrix} 2\zeta+1 & \zeta -1 \\ 2\zeta-2 & \zeta+2 \end{pmatrix} \;. 
\]
There are two conjugacy classes of reflections: the one of $s$ and the one of $s^2$ (both of order $3$ and length $4$).
Hence, there is just one orbit $\Omega_1$ of reflection hyperplanes and so the parameters for the restricted rational Cherednik algebra are $k \dopgleich (k_{1,0},k_{1,1}, k_{1,2})$. We can now compute that
\[ 
\begin{array}{lcl}
c_k(s) & = & (-\zeta-2)k_{1,0} + (-\zeta+1)k_{1,1} + (2 \zeta+1)k_{1,2} \;, \\
c_k(s^2) & = &(\zeta-1)k_{1,0} + (\zeta+2)k_{1,1} + (-2\zeta-1)k_{1,2} \;.
\end{array}
\]
The group $G_4$ has 7 irreducible characters and from the character table we can easily compute the table \ref{g4_data}.
\begin{table}[H]
\centering
\footnotesize
\begin{minipage}[t]{0.45\linewidth}
\centering
\begin{tabular}[t]{|c|c|c|c|} 
\hline
$\lambda$ & $\lambda(s)$ & $\lambda(s^2)$ & $\omega_{\lambda}(\check{\mrm{eu}}_k)$ \\ \hline \hline
$\phi_{1,0}$ & $1$ & $1$ & $12k_{1,0} - 12k_{1,1}$ \\ \hline
$\phi_{1,4}$ & $\zeta$ & $-\zeta - 1$ & $0$ \\ \hline
$\phi_{1,8}$ & $-\zeta - 1$ & $\zeta$ & $12k_{1,0} - 12k_{1,2}$ \\ \hline
$\phi_{2,5}$ & $-1$ & $-1$ & $6k_{1,0} - 6k_{1,2}$ \\ \hline
\end{tabular}
\end{minipage}
\begin{minipage}[t]{0.45\linewidth}
\centering
\begin{tabular}[t]{|c|c|c|c|} 
\hline
$\lambda$ & $\lambda(s)$ & $\lambda(s^2)$ & $\omega_{\lambda}(\check{\mrm{eu}}_k)$ \\ \hline \hline
$\phi_{2,3}$ & $-\zeta$ & $\zeta + 1$ & $12k_{1,0} - 6k_{1,1} - 6k_{1,2}$ \\ \hline
$\phi_{2,1}$ & $\zeta + 1$ & $-\zeta$ & $6k_{1,0} - 6k_{1,1}$ \\ \hline
$\phi_{3,2}$ & $0$ & $0$ & $8k_{1,0} - 4k_{1,1} - 4k_{1,2}$  \\ \hline
\end{tabular} 
\end{minipage}
\caption{Character data for $G_4$.} \label{g4_data}
\end{table} \vspace{-10pt}

\noindent This table immediately reveals that the condition $\chi_\mu(1) \lambda(s) = \lambda(1)\chi_\mu(s)$ is \textit{never} satisfied for $\lambda \neq \mu$, implying that the generic Euler families are singletons. But as $\mrm{CM}_k$ is a refinement of $\mrm{Eu}_k$, we must already have 
\[
\mrm{Eu}_k = \mrm{CM}_k \quad \tn{for all} \quad k \notin \Phi^{-1}\mrm{EuEx}(G_4) \;.
\]

By considering non-zero differences $\omega_\lambda(\check{\mrm{eu}}_k) - \omega_\mu(\check{\mrm{eu}}_k)$ we can furthermore compute from table \ref{g4_data} that $\Phi^{-1}\mrm{EuEx}(G_{4})$ is the union of the following six hyperplanes, which form two orbits under the symmetric group $\rS_3$:
\begin{table}[H]
\centering
\footnotesize
\begin{tabular}{|l|l|l|}
\hline
Label & Orbit & Length \\
\hline \hline
$1$ & $\rS_3.(0,1,-1)$ & $3$ \\
\hline
$2$ & $\rS_3.(1,-1,1)$ & $3$ \\
\hline
\end{tabular}
\caption{$\Phi^{-1}\mrm{EuEx}(G_4)$.}
\end{table}

\noindent The operation $\sharp$ is described by the action of the cycle $(2,3) \in \rS_3$ so that all orbits and thus $\Phi^{-1}\mrm{EuEx}(G_4)$ are stable under $\sharp$. A comparison with the data in \cite{Chlouveraki.MGAP-functions-for-th} shows that the hyperplanes above are precisely the six essential hyperplanes for $G_4$ so that $\mrm{RouEx}(G_4) = \Phi^{-1}\mrm{EuEx}(G_4)$ and that the generic Rouquier families are also singletons. This proves \ref{main_theorem_2} for $G_4$.
\end{paran}

\begin{remark}
We emphasize that Bellamy already proved in \cite{Bel-Singular-CM} that the Calogero--Moser families for $G_4$ are singletons for generic parameters. His proof was more involved however and did not directly give an explicit open subset of \textit{generic} parameters. 
\end{remark}

\begin{remark}
More elaborate computational methods have been used in \cite{Thiel.UOn-restricted-ration} to compute the Calogero--Moser families on each of the hyperplanes of $\Phi^{-1}\mrm{EuEx}(G_4)$. It turns out that $\mrm{RouEx}(G_4) = \Phi^{-1}\mrm{EuEx}(G_4) = \Phi^{-1}\mrm{CMEx}(G_4)$ and that on each of the hyperplanes the Calogero--Moser families coincide with the Rouquier families. This indeed confirms Martino's conjecture in its full form for $G_4$. Due to the additional computational ingredients involved in this approach, we will postpone the discussion to a future article.
\end{remark}

\vspace{-10pt}
\subsection{$G_{25}$}
\begin{paran}
The group $G_{25}$ is of order 648 and can be realized as the matrix group over $\bbQ(\zeta)$, where $\zeta \dopgleich \zeta_3$, generated by the reflections
\[
s \dopgleich \begin{pmatrix} 1 & 0 & 0 \\ 0 & 1 & 0 \\ 0 & 0 & \zeta \end{pmatrix}, \quad t \dopgleich \frac{1}{3} \begin{pmatrix} \zeta+2 & \zeta-1 & \zeta-1 \\ \zeta-1 & \zeta+2 & \zeta-1 \\ \zeta-1 & \zeta-1 & \zeta+2 \end{pmatrix}, \quad u \dopgleich \begin{pmatrix} 1 & 0 & 0 \\ 0 & \zeta & 0 \\ 0 & 0 & 1 \end{pmatrix} \;.
\]
There are two conjugacy classes of reflections: the one of $s$ and the one of $s^2$ (both of order $3$ and length $12$). As in \S\ref{g4} we deduce that 
the parameters for the restricted rational Cherednik algebra are $k \dopgleich (k_{1,0},k_{1,1}, k_{1,2})$ and we can now compute that
\[
\begin{array}{lcl}
c_k(s) &=& (-\zeta-2)k_{1,0} + (-\zeta+1)k_{1,1} + (2\zeta+1)k_{1,2} \;, \\
c_k(s^2) &=& (\zeta-1)k_{1,0} + (\zeta+2)k_{1,1} + (-2\zeta-1)k_{1,2} \;.
\end{array}
\]
The group $G_{25}$ has 24 irreducible characters and from the character table we can compute the table \ref{g25_data} on page \pageref{g25_data}.
\begin{table}[htbp] 
\begin{minipage}[t]{0.49\linewidth}
\scriptsize
\begin{tabular}[t]{|c|c|c|c|c|}
\hline
$\lambda$ & $\lambda(s)$ & $\lambda(s^2)$ & $\omega_{\lambda}(\check{\mrm{eu}}_k)$ & ss \\ \hline \hline
$\phi_{1,0}$ & $1$ & $1$ & $36k_{1,0} - 36k_{1,1}$ & n \\ \hline
$\phi_{1,24}$ & $-\zeta - 1$ & $\zeta$ & $36k_{1,0} - 36k_{1,2}$ & n \\ \hline
$\phi_{1,12}$ & $\zeta$ & $-\zeta - 1$ & $0$ & n \\ \hline
$\phi_{2,15}$ & $-1$ & $-1$ & $18k_{1,0} - 18k_{1,2}$ & n \\ \hline
$\phi_{2,9}$ & $-\zeta$ & $\zeta + 1$ & $36k_{1,0} - 18k_{1,1} - 18k_{1,2}$ & n \\ \hline
$\phi_{2,3}$ & $\zeta + 1$ & $-\zeta$ & $18k_{1,0} - 18k_{1,1}$ & n \\ \hline
$\phi_{3,6}$ & $0$ & $0$ & $24k_{1,0} - 12k_{1,1} - 12k_{1,2}$ & y \\ \hline
$\phi_{3,1}$ & $\zeta + 2$ & $-\zeta + 1$ & $24k_{1,0} - 24k_{1,1}$ & n \\ \hline
$\phi_{3,5}'$ & $-\zeta + 1$ & $\zeta + 2$ & $36k_{1,0} - 24k_{1,1} - 12k_{1,2}$ & n \\ \hline
$\phi_{3,13}''$ & $-2\zeta - 1$ & $2\zeta + 1$ & $36k_{1,0} - 12k_{1,1} - 24k_{1,2}$ & n \\ \hline
$\phi_{3,17}$ & $-\zeta - 2$ & $\zeta - 1$ & $24k_{1,0} - 24k_{1,2}$ & n \\ \hline
$\phi_{3,5}''$ & $2\zeta + 1$ & $-2\zeta - 1$ & $12k_{1,0} - 12k_{1,1}$ & n \\ \hline
\end{tabular}
\end{minipage}
\begin{minipage}[t]{0.48\linewidth}
\scriptsize
\begin{tabular}[t]{|c|c|c|c|c|}
\hline
$\lambda$ & $\lambda(s)$ & $\lambda(s^2)$ & $\omega_{\lambda}(\check{\mrm{eu}}_k)$ & ss \\ \hline \hline
$\phi_{3,13}'$ & $\zeta - 1$ & $-\zeta - 2$ & $12k_{1,0} - 12k_{1,2}$ & n \\ \hline
$\phi_{6,4}''$ & $2\zeta + 1$ & $-2\zeta - 1$ & $18k_{1,0} - 12k_{1,1} - 6k_{1,2}$ & n \\ \hline
$\phi_{6,10}$ & $-\zeta - 2$ & $\zeta - 1$ & $24k_{1,0} - 6k_{1,1} - 18k_{1,2}$ & n \\ \hline
$\phi_{6,8}'$ & $\zeta - 1$ & $-\zeta - 2$ & $18k_{1,0} - 6k_{1,1} - 12k_{1,2}$ & n \\ \hline
$\phi_{6,2}$ & $\zeta + 2$ & $-\zeta + 1$ & $24k_{1,0} - 18k_{1,1} - 6k_{1,2}$ & n \\ \hline
$\phi_{6,8}''$ & $-2\zeta - 1$ & $2\zeta + 1$ & $30k_{1,0} - 12k_{1,1} - 18k_{1,2}$ & n \\ \hline
$\phi_{6,4}'$ & $-\zeta + 1$ & $\zeta + 2$ & $30k_{1,0} - 18k_{1,1} - 12k_{1,2}$ & n \\ \hline
$\phi_{8,3}$ & $2$ & $2$ & $27k_{1,0} - 18k_{1,1} - 9k_{1,2}$ & n \\ \hline
$\phi_{8,9}$ & $-2\zeta - 2$ & $2\zeta$ & $27k_{1,0} - 9k_{1,1} - 18k_{1,2}$ & n \\ \hline
$\phi_{8,6}$ & $2\zeta$ & $-2\zeta - 2$ & $18k_{1,0} - 9k_{1,1} - 9k_{1,2}$ & n \\ \hline
$\phi_{9,7}$ & $0$ & $0$ & $24k_{1,0} - 12k_{1,1} - 12k_{1,2}$ & y \\ \hline
$\phi_{9,5}$ & $0$ & $0$ & $24k_{1,0} - 12k_{1,1} - 12k_{1,2}$ & y \\ \hline
\end{tabular}
\end{minipage} \caption{Character data for $G_{25}$.} \label{g25_data}
\end{table}
We can immediately see in this table that the three characters $\phi_{3,6}$, $\phi_{9,7}$, and $\phi_{9,5}$, lie in a common Euler $k$-family for any $k$ as the character values on reflections is always zero. A consideration of the values $\omega_\lambda(\check{\mrm{eu}}_k)$ listed in this table shows that the family $\lbrace \phi_{3,6}, \phi_{9,7}, \phi_{9,5} \rbrace$ indeed forms a generic Euler family and furthermore that this is the only non-singleton generic Euler family. All characters different from $\phi_{3,6}$, $\phi_{9,7}$, and $\phi_{9,5}$, thus form a singleton Calogero--Moser $k$-family for all $k \notin \Phi^{-1}\mrm{EuEx}(G_{25})$ and it remains to decide whether the generic Euler family $\lbrace \phi_{3,6}, \phi_{9,7}, \phi_{9,5} \rbrace$ is already a Calogero--Moser family or if this family splits into several Calogero--Moser families. The concept of supersingularity resolves this question. Namely, in the column denoted by \textit{ss} in table \ref{g25_data} we have listed if the character is supersingular (symbolized by \textit{y}) or not (symbolized by \textit{n}), and we see that the three characters $\phi_{3,6}$, $\phi_{9,7}$, and $\phi_{9,5}$ are supersingular. So, these three characters form a good generic Euler family and now we know from \ref{trick} that this is indeed a Calogero--Moser $k$-family 
for $k \notin \Phi^{-1}\mrm{EuEx}(G_{25})$. Hence,
\[
\mrm{Eu}_k = \mrm{CM}_k \quad \tn{for all} \quad k \notin \Phi^{-1}\mrm{EuEx}(G_{25}) \;.
\]

From table \ref{g25_data} we can also compute that $\Phi^{-1}\mrm{EuEx}(G_{25})$ is the union of the following 30 hyperplanes, which form six orbits under the group $\rS_3$:
\begin{table}[H]
\centering
\footnotesize
\begin{minipage}[t]{0.35\linewidth}
\centering
\begin{tabular}{|l|l|l|}
\hline
Label & Orbit & Length \\
\hline \hline
$1$ & $\rS_3.(0, 1, -1)$ & $3$ \\\hline
$2$ & $\rS_3.(1, -6, 5)$ & $6$ \\\hline
$3$ & $\rS_3.(1, -4, 3)$ & $6$ \\\hline
\end{tabular}
\end{minipage}
\begin{minipage}[t]{0.35\linewidth}
\centering
\begin{tabular}{|l|l|l|}
\hline
Label & Orbit & Length \\
\hline \hline
$4$ & $\rS_3.(1, -3, 2)$ & $6$ \\\hline
$5$ & $\rS_3.(1, -2, 1)$ & $3$ \\\hline
$6$ & $\rS_3.(2, -5, 3)$ & $6$ \\ \hline
\end{tabular}
\end{minipage} \caption{$\Phi^{-1}\mrm{EuEx}(G_{25})$.}
\end{table}

\noindent Again the operation $\sharp$ is described by the action of the cycle $(2,3) \in \rS_3$ so that all orbits and thus $\Phi^{-1}\mrm{EuEx}(G_{25})$ are stable under $\sharp$. A comparison with the data in \cite{Chlouveraki.MGAP-functions-for-th} shows that the 12 hyperplanes in the orbits $1$, $4$, and $5$ are precisely the essential hyperplanes for $G_{25}$ so that $\mrm{RouEx}(G_{25}) \subs \Phi^{-1}\mrm{EuEx}(G_{25})$. The data in \cite{Chlouveraki.MGAP-functions-for-th} shows however that although there is also just one non-singleton generic Rouquier family, this family is equal to $\lbrace \phi_{9,7},\phi_{9,5} \rbrace$. Hence, in contrast to the generic Calogero--Moser families, the character $\phi_{3,6}$ lies in a singleton generic Rouquier family and this shows that the generic Calogero--Moser families are \textit{unions} of generic Rouquier families but not equal to them. This still proves Martino's special parameter conjecture for all $k \in \Phi^{-1}\mrm{CMGen}(G_{25})$ but it disproves Martino's generic parameter conjecture!
\qed
\end{paran}

\subsection{$G_5$}
\begin{paran}
The group $G_5$ is of order 72 and can be realized as the matrix group over $\bbQ(\zeta)$, where $\zeta \dopgleich \zeta_3$, generated by the reflections
\[
s \dopgleich \begin{pmatrix} 1 & 0 \\ 0 & \zeta \end{pmatrix}, \quad t \dopgleich \frac{1}{3} \begin{pmatrix} \zeta+2 & -\zeta+1 \\ -2\zeta+2 & 2\zeta+1 \end{pmatrix} \;. 
\]
There are four conjugacy classes of reflections: the one of $s$, of $s^2$, of $t$, and of $t^2$ (all of order $3$ and length $4$). 
Hence, there are two orbits of reflection hyperplanes, namely the orbit $\Omega_1$ of the reflection hyperplane of $s$, and the orbit $\Omega_2$ of the reflection hyperplane of $t$. The parameters for the restricted rational Cherednik algebra are therefore $k \dopgleich (k_{1,0},k_{1,1}, k_{1,2}, k_{2,0}, k_{2,1},k_{2,2})$ and we can now compute that
\[ 
\begin{array}{lcl}
c_k(s) &= &(-\zeta - 2)k_{1,0} + (-\zeta + 1)k_{1,1} + (2\zeta + 1)k_{1,2} \;, \\
c_k(s^2) &= &(\zeta - 1)k_{1,0} + (\zeta + 2)k_{1,1} + (-2\zeta - 1)k_{1,2}\;, \\
c_k(t) &=& (-\zeta - 2)k_{2,0} + (-\zeta + 1)k_{2,1} + (2\zeta + 1)k_{2,2} \;, \\
c_k(t^2) &=& (\zeta - 1)k_{2,0} + (\zeta + 2)k_{2,1} + (-2\zeta - 1)k_{2,2} \;.
\end{array}
\]
The group $G_5$ has 21 irreducible characters and from the character table of $G_5$ we can compute the table \ref{g5_data}.
\begin{table}[H] 
\centering
\footnotesize
\begin{tabular}{|c|c|c|c|c|}
\hline
$\lambda$ & $\lambda(s)$ & $\lambda(s^2)$ & $\omega_{\lambda}(\check{\mrm{eu}}_k)$ & ss \\ \hline \hline
$\phi_{1,0}$ & $1$ & $1$ & $12k_{1,0} - 12k_{1,1} + 12k_{2,0} - 12k_{2,1}$ & n \\ \hline
$\phi_{1,12}''$ & $-\zeta - 1$ & $\zeta$ & $12k_{1,0} - 12k_{1,2}$ & n \\ \hline
$\phi_{1,16}$ & $-\zeta - 1$ & $-\zeta - 1$ & $12k_{1,0} - 12k_{1,2} + 12k_{2,0} - 12k_{2,2}$ & n \\ \hline
$\phi_{1,4}'$ & $1$ & $\zeta$ & $12k_{1,0} - 12k_{1,1}$ & n \\ \hline
$\phi_{1,8}''$ & $\zeta$ & $\zeta$ & $0$ & n \\ \hline
$\phi_{1,8}'$ & $1$ & $-\zeta - 1$ & $12k_{1,0} - 12k_{1,1} + 12k_{2,0} - 12k_{2,2}$ & n \\ \hline
$\phi_{1,8}'''$ & $-\zeta - 1$ & $1$ & $12k_{1,0} - 12k_{1,2} + 12k_{2,0} - 12k_{2,1}$ & n \\ \hline
$\phi_{1,4}''$ & $\zeta$ & $1$ & $12k_{2,0} - 12k_{2,1}$ & n \\ \hline
$\phi_{1,12}'$ & $\zeta$ & $-\zeta - 1$ & $12k_{2,0} - 12k_{2,2}$ & n \\ \hline
$\phi_{2,9}$ & $-1$ & $-1$ & $6k_{1,0} - 6k_{1,2} + 6k_{2,0} - 6k_{2,2}$ & n \\ \hline
$\phi_{2,7}''$ & $-\zeta$ & $-1$ & $12k_{1,0} - 6k_{1,1} - 6k_{1,2} + 6k_{2,0} - 6k_{2,2}$ & n \\ \hline
$\phi_{2,3}'$ & $-\zeta$ & $\zeta + 1$ & $12k_{1,0} - 6k_{1,1} - 6k_{1,2} + 6k_{2,0} - 6k_{2,1}$ & n \\ \hline
$\phi_{2,5}'''$ & $\zeta + 1$ & $-1$ & $6k_{1,0} - 6k_{1,1} + 6k_{2,0} - 6k_{2,2}$ & n \\ \hline
$\phi_{2,3}''$ & $\zeta + 1$ & $-\zeta$ & $6k_{1,0} - 6k_{1,1} + 12k_{2,0} - 6k_{2,1} - 6k_{2,2}$ & n \\ \hline
$\phi_{2,5}''$ & $-\zeta$ & $-\zeta$ & $12k_{1,0} - 6k_{1,1} - 6k_{1,2} + 12k_{2,0} - 6k_{2,1} - 6k_{2,2}$ & n \\ \hline
$\phi_{2,1}$ & $\zeta + 1$ & $\zeta + 1$ & $6k_{1,0} - 6k_{1,1} + 6k_{2,0} - 6k_{2,1}$ & n \\ \hline
$\phi_{2,7}'$ & $-1$ & $-\zeta$ & $6k_{1,0} - 6k_{1,2} + 12k_{2,0} - 6k_{2,1} - 6k_{2,2}$ & n \\ \hline
$\phi_{2,5}'$ & $-1$ & $\zeta + 1$ & $6k_{1,0} - 6k_{1,2} + 6k_{2,0} - 6k_{2,1}$ & n \\ \hline
$\phi_{3,6}$ & $0$ & $0$ & $8k_{1,0} - 4k_{1,1} - 4k_{1,2} + 8k_{2,0} - 4k_{2,1} - 4k_{2,2}$ & y \\ \hline
$\phi_{3,4}$ & $0$ & $0$ & $8k_{1,0} - 4k_{1,1} - 4k_{1,2} + 8k_{2,0} - 4k_{2,1} - 4k_{2,2}$ & y \\ \hline
$\phi_{3,2}$ & $0$ & $0$ & $8k_{1,0} - 4k_{1,1} - 4k_{1,2} + 8k_{2,0} - 4k_{2,1} - 4k_{2,2}$ & y \\ \hline
\end{tabular}
\caption{Character data for $G_{5}$.} \label{g5_data}
\end{table}

\noindent From this table we see that there is only one non-singleton generic Euler family, namely $\lbrace \phi_{3,6}, \phi_{3,4}, \phi_{3,2} \rbrace$. Furthermore, we see that these three characters are all supersingular so that this is a good generic Euler family and thus already a Calogero--Moser $k$-family for all $k \notin \Phi^{-1}\mrm{EuEx}(G_5)$. 
Hence,
\[
\mrm{Eu}_k = \mrm{CM}_k \quad \tn{for all} \quad k \notin \Phi^{-1}\mrm{EuEx}(G_5) \;.
\]

From table \ref{g5_data} we can also compute that $\Phi^{-1}\mrm{EuEx}(\Gamma)$ is the union of the following 69 hyperplanes, which form six orbits under the Young subgroup $\Sigma \dopgleich \Sigma_{\lbrace \lbrace 1,2,3 \rbrace, \lbrace 4,5,6 \rbrace \rbrace}$ of $\rS_6$:
\begin{table}[H]
\centering
\footnotesize
\begin{minipage}[t]{0.42\linewidth}
\centering
\begin{tabular}{|l|l|l|}
\hline
Label & Orbit & Length \\
\hline \hline
$1$a & $\Sigma.(0, 1, -1, 0, 0, 0)$ & $3$ \\ \hline
$1$b & $\Sigma.(0, 0, 0, 0, 1, -1)$ & $3$ \\ \hline
$2$ & $\Sigma.(0, 1, -1, 0, -1, 1)$ & $18$ \\ \hline
\end{tabular}
\end{minipage}
\begin{minipage}[t]{0.42\linewidth}
\centering
\begin{tabular}{|l|l|l|}
\hline
Label & Orbit & Length \\
\hline \hline
$3$a & $\Sigma.(0, 1, -1, -2, 1, 1)$ & $18$ \\ \hline
$3$b & $\Sigma.(1, -2, 1, 0, -1, 1)$ & $18$ \\ \hline
$4$ & $\Sigma.(1, -2, 1, -2, 1, 1)$ & $9$ \\ \hline
\end{tabular} 
\end{minipage} \caption{$\Phi^{-1}\mrm{EuEx}(G_5)$.}
\end{table} 

\noindent The operation $\sharp$ is described by the action of the permutation $(2,3)(5,6) \in \Sigma$ so that all orbits and thus $\Phi^{-1}\mrm{EuEx}(G_5)$ are stable under $\sharp$. A comparison with the data in \cite{Chlouveraki.MGAP-functions-for-th} shows that the 24 hyperplanes in the orbits $1$a, $1$b, $4$, and in the suborbit $\langle (1,2,3), (4,5,6) \rangle.(0,1,-1,0,-1,1)$ of orbit $2$ are precisely the essential hyperplanes of $G_5$ so that $\mrm{RouEx}(G_5) \subs \Phi^{-1}\mrm{EuEx}(G_5)$. Furthermore, the data in \cite{Chlouveraki.MGAP-functions-for-th} shows that the generic Rouquier families coincide with the generic Calogero--Moser families just determined. This proves \ref{main_theorem_2} for $G_5$. \qed
\end{paran}

\subsection{$G_6$}

\begin{paran}
The group $G_6$ is of order 48 and can be realized as the matrix group over $\bbQ(\zeta)$, where $\zeta \dopgleich \zeta_{12}$, generated by the reflections
\[
s \dopgleich \frac{1}{3} \begin{pmatrix} -\zeta^3 + 2\zeta & -\zeta^3 + 2\zeta \\ -2\zeta^3 + 4\zeta & \zeta^3 - 2\zeta \end{pmatrix}, \quad t \dopgleich \begin{pmatrix}  1 & 0 \\ 0 & \zeta^2 -1 \end{pmatrix} \;.
\]
There are three conjugacy classes of reflections: the one of $s$ (of order $2$ and length $4$), and the ones of $t$ and $t^2$ (both of order $3$ and length $4$). 
Hence, there are two orbits of reflection hyperplanes, namely the orbit $\Omega_1$ of the reflection hyperplane of $s$, and the orbit $\Omega_2$ of the reflection hyperplane of $t$. The parameters for the restricted rational Cherednik algebra are therefore $k \dopgleich (k_{1,0},k_{1,1}, k_{2,0}, k_{2,1},k_{2,2})$ and we can now compute that
\[ 
\begin{array}{lcl}
c_k(s) & = & -2k_{1,0} + 2k_{1,1} \;, \\
c_k(t) & = & (-\zeta^2 - 1)k_{2,0} + (-\zeta^2 + 2)k_{2,1} + (2\zeta^2 - 1)k_{2,2} \;, \\
c_k(t^2) & = & (\zeta^2 - 2)k_{2,0} + (\zeta^2 + 1)k_{2,1} + (-2\zeta^2 + 1)k_{2,2} \;.
\end{array}
\]
The group $G_6$ has 14 irreducible characters and from the character table we can compute the table \ref{g6_data}.
\begin{table}[H] 
\centering
\scriptsize
\begin{tabular}{|c|c|c|c|c|c|}
\hline
$\lambda$ & $\lambda(s)$ & $\lambda(t)$ & $\lambda(t^2)$ & $\omega_{\lambda}(\check{\mrm{eu}}_k)$ & ss \\ \hline \hline
$\phi_{1,0}$ & $1$ & $1$ & $1$ & $12k_{1,0} - 12k_{1,1} + 12k_{2,0} - 12k_{2,1}$ & n \\ \hline
$\phi_{1,4}$ & $1$ & $\zeta^2 - 1$ & $-\zeta^2$ & $12k_{1,0} - 12k_{1,1}$ & n \\ \hline
$\phi_{1,8}$ & $1$ & $-\zeta^2$ & $\zeta^2 - 1$ & $12k_{1,0} - 12k_{1,1} + 12k_{2,0} - 12k_{2,2}$ & n \\ \hline
$\phi_{1,6}$ & $-1$ & $1$ & $1$ & $12k_{2,0} - 12k_{2,1}$ & n \\ \hline
$\phi_{1,10}$ & $-1$ & $\zeta^2 - 1$ & $-\zeta^2$ & $0$ & n \\ \hline
$\phi_{1,14}$ & $-1$ & $-\zeta^2$ & $\zeta^2 - 1$ & $12k_{2,0} - 12k_{2,2}$ & n \\ \hline
$\phi_{2,5}''$ & $0$ & $-1$ & $-1$ & $6k_{1,0} - 6k_{1,1} + 6k_{2,0} - 6k_{2,2}$ & y \\ \hline
$\phi_{2,3}''$ & $0$ & $-\zeta^2 + 1$ & $\zeta^2$ & $6k_{1,0} - 6k_{1,1} + 12k_{2,0} - 6k_{2,1} - 6k_{2,2}$ & y \\ \hline
$\phi_{2,3}'$ & $0$ & $\zeta^2$ & $-\zeta^2 + 1$ & $6k_{1,0} - 6k_{1,1} + 6k_{2,0} - 6k_{2,1}$ & y \\ \hline
$\phi_{2,7}$ & $0$ & $-1$ & $-1$ & $6k_{1,0} - 6k_{1,1} + 6k_{2,0} - 6k_{2,2}$ & y \\ \hline
$\phi_{2,1}$ & $0$ & $\zeta^2$ & $-\zeta^2 + 1$ & $6k_{1,0} - 6k_{1,1} + 6k_{2,0} - 6k_{2,1}$ & y \\ \hline
$\phi_{2,5}'$ & $0$ & $-\zeta^2 + 1$ & $\zeta^2$ & $6k_{1,0} - 6k_{1,1} + 12k_{2,0} - 6k_{2,1} - 6k_{2,2}$ & y \\ \hline
$\phi_{3,2}$ & $1$ & $0$ & $0$ & $8k_{1,0} - 8k_{1,1} + 8k_{2,0} - 4k_{2,1} - 4k_{2,2}$ & n \\ \hline
$\phi_{3,4}$ & $-1$ & $0$ & $0$ & $4k_{1,0} - 4k_{1,1} + 8k_{2,0} - 4k_{2,1} - 4k_{2,2}$ & n \\ \hline
\end{tabular}
\caption{Character data for $G_{6}$.} \label{g6_data}
\end{table}
\noindent From this table we see that the only non-singleton generic Euler families are
\[
\lbrace \phi_{2,5}'', \phi_{2,7} \rbrace, \; \lbrace \phi_{2,3}'', \phi_{2,5}' \rbrace, \; \lbrace \phi_{2,3}', \phi_{2,1} \rbrace 
\]
and that all these characters are supersingular. Hence, these generic Euler families are good and thus already Calogero--Moser $k$-families for all $k \notin \Phi^{-1}\mrm{EuEx}(G_6)$ by \ref{trick}. This shows that
\[
\mrm{Eu}_k = \mrm{CM}_k \quad \tn{for all} \quad k \notin \Phi^{-1}\mrm{EuEx}(\Gamma) \;.
\]

From table \ref{g6_data} we can also compute that $\Phi^{-1}\mrm{EuEx}(G_6)$ is the union of the following 22 hyperplanes, which form five orbits under the Young subgroup $\Sigma \dopgleich \Sigma_{ \lbrace \lbrace 1,2 \rbrace, \lbrace 3,4,5 \rbrace \rbrace}$ of $\rS_5$: 
\begin{table}[H]
\centering
\footnotesize
\begin{minipage}[t]{0.4\linewidth}
\centering
\begin{tabular}[t]{|l|l|l|}
\hline
Label & Orbit & Length \\
\hline \hline
$1$a & $\Sigma.(0,0,0,1,-1)$ & $3$ \\\hline
$1$b & $\Sigma.(1,-1,0,0,0)$ & $1$ \\\hline
$2$ & $\Sigma.(1,-1,0,-1,1)$ & $6$ \\\hline
\end{tabular}
\end{minipage}
\begin{minipage}[t]{0.4\linewidth}
\centering
\begin{tabular}[t]{|l|l|l|}
\hline
Label & Orbit & Length \\
\hline \hline
$3$ & $\Sigma.(1,-1,-2,1,1)$ & $6$ \\\hline
$4$ & $\Sigma.(2,-2,-2,1,1)$ & $6$ \\ \hline
\end{tabular}
\end{minipage} \caption{$\Phi^{-1}\mrm{EuEx}(G_6)$.}
\end{table} 

\noindent The operation $\sharp$ is described by the action of the cycle $(4,5) \in \Sigma$ so that all orbits and thus $\Phi^{-1}\mrm{EuEx}(G_6)$ are stable under $\sharp$. A comparison with the data in \cite{Chlouveraki.MGAP-functions-for-th} shows that the 16 hyperplanes in orbits $1$a, $1$b, $2$, and $3$ are precisely the essential hyperplanes of $G_6$. Furthermore, the generic Rouquier families coincide with the generic Calogero--Moser families. This proves \ref{main_theorem_2} for $G_6$. \qed
\end{paran}

\subsection{$G_8$}

\begin{paran}
The group $G_8$ is of order 96 and can be realized as the matrix group over $\bbQ(\zeta)$, $\zeta \dopgleich \zeta_4$, generated by the reflections 
\[
s \dopgleich \begin{pmatrix} 1 & 0 \\ 0 & \zeta \end{pmatrix}, \quad t \dopgleich \frac{1}{2} \begin{pmatrix} \zeta+1 & \zeta-1 \\ \zeta-1 & \zeta+1 \end{pmatrix} \;. 
\]
There are three conjugacy classes of reflections: the ones of $s$ and $s^2$ (both of order $4$ and length $6$), and the one of $s^3$ (of order $2$ and length $6$). Hence, there is just one orbit of reflection hyperplanes, namely the orbit $\Omega_1$ of the reflection hyperplane of $s$. The parameters for the restricted rational Cherednik algebra are therefore $(k_{1,j})_{0 \leq j \leq 3}$ and we can now compute that
\[ 
\begin{array}{lcl}
c_k(s) & = & (-\zeta - 1)k_{1,0} + (-\zeta + 1)k_{1,1} + (\zeta + 1)k_{1,2} + (\zeta - 1)k_{1,3} \;, \\
c_k(s^2) & = & -2k_{1,0} + 2k_{1,1} - 2k_{1,2} + 2k_{1,3} \;, \\
c_k(s^3) & = & (\zeta - 1)k_{1,0} + (\zeta + 1)k_{1,1} + (-\zeta + 1)k_{1,2} + (-\zeta - 1)k_{1,3} \;.
\end{array}
\]
The group $G_8$ has 16 irreducible characters and from the character table of $G_8$ we can compute the table \ref{g8_data}.
\begin{table}[H] 
\centering
\footnotesize
\begin{tabular}{|c|c|c|c|c|c|}
\hline
$\lambda$ & $\lambda(s)$ & $\lambda(s^2)$ & $\lambda(s^3)$ & $\omega_{\lambda}(\check{\mrm{eu}}_k)$ & ss \\ \hline \hline
$\phi_{1,0}$ & $1$ & $1$ & $1$ & $24k_{1,0} - 24k_{1,1}$ & n \\ \hline
$\phi_{1,6}$ & $\zeta$ & $-1$ & $-\zeta$ & $0$ & n \\ \hline
$\phi_{1,12}$ & $-1$ & $1$ & $-1$ & $24k_{1,0} - 24k_{1,3}$ & n \\ \hline
$\phi_{1,18}$ & $-\zeta$ & $-1$ & $\zeta$ & $24k_{1,0} - 24k_{1,2}$ & n \\ \hline
$\phi_{2,1}$ & $\zeta + 1$ & $0$ & $-\zeta + 1$ & $12k_{1,0} - 12k_{1,1}$ & n \\ \hline
$\phi_{2,4}$ & $0$ & $2$ & $0$ & $24k_{1,0} - 12k_{1,1} - 12k_{1,3}$ & n \\ \hline
$\phi_{2,7}'$ & $-\zeta + 1$ & $0$ & $\zeta + 1$ & $24k_{1,0} - 12k_{1,1} - 12k_{1,2}$ & n \\ \hline
$\phi_{2,7}''$ & $\zeta - 1$ & $0$ & $-\zeta - 1$ & $12k_{1,0} - 12k_{1,3}$ & n \\ \hline
$\phi_{2,10}$ & $0$ & $-2$ & $0$ & $12k_{1,0} - 12k_{1,2}$ & n \\ \hline
$\phi_{2,13}$ & $-\zeta - 1$ & $0$ & $\zeta - 1$ & $24k_{1,0} - 12k_{1,2} - 12k_{1,3}$ & n \\ \hline
$\phi_{3,8}$ & $-1$ & $-1$ & $-1$ & $16k_{1,0} - 8k_{1,2} - 8k_{1,3}$ & n \\ \hline
$\phi_{3,6}$ & $-\zeta$ & $1$ & $\zeta$ & $24k_{1,0} - 8k_{1,1} - 8k_{1,2} - 8k_{1,3}$ & n \\ \hline
$\phi_{3,4}$ & $1$ & $-1$ & $1$ & $16k_{1,0} - 8k_{1,1} - 8k_{1,2}$ & n \\ \hline
$\phi_{3,2}$ & $\zeta$ & $1$ & $-\zeta$ & $16k_{1,0} - 8k_{1,1} - 8k_{1,3}$ & n \\ \hline
$\phi_{4,5}$ & $0$ & $0$ & $0$ & $18k_{1,0} - 6k_{1,1} - 6k_{1,2} - 6k_{1,3}$ & y \\ \hline
$\phi_{4,3}$ & $0$ & $0$ & $0$ & $18k_{1,0} - 6k_{1,1} - 6k_{1,2} - 6k_{1,3}$ & y \\ \hline
\end{tabular}
\caption{Character data for $G_8$.} \label{g8_data}
\end{table}

\noindent In this table we see that there is only one non-singleton generic Euler family, namely $\lbrace \phi_{4,3}, \phi_{4,5} \rbrace$, and that this family is good. Hence, it is already a Calogero--Moser $k$-family for all $k \notin \Phi^{-1}\mrm{EuEx}(G_8)$ and therefore
\[
\mrm{Eu}_k = \mrm{CM}_k \quad \tn{for all } \quad k \notin \Phi^{-1}\mrm{EuEx}(G_8) \;.
\]

From table \ref{g8_data} we can also compute that $\Phi^{-1}\mrm{EuEx}(G_8)$ is the union of the following 37 hyperplanes, which form five orbits under the group $\rS_4$:
\begin{table}[H]
\centering
\footnotesize
\begin{minipage}[t]{0.38\linewidth}
\centering
\begin{tabular}[t]{|l|l|l|}
\hline
Label & Orbit & Length \\
\hline \hline
$1$ & $\rS_4.(0,0,1,-1)$ & 6 \\ \hline
$2$ & $\rS_4.(0,1,-2,1)$ & 12 \\ \hline
$3$ & $\rS_4.(1,-3,1,1)$ & 4 \\ \hline
\end{tabular}
\end{minipage}
\begin{minipage}[t]{0.38\linewidth}
\centering
\begin{tabular}[t]{|l|l|l|}
\hline
Label & Orbit & Length \\
\hline \hline
$4$ & $\rS_4.(1,-2,-2,3)$ & 12 \\ \hline
$5$ & $\rS_4.(1,-1,-1,1)$ & 3 \\ \hline
\end{tabular}
\end{minipage} \caption{$\Phi^{-1}\mrm{EuEx}(G_8)$.}
\end{table}

\noindent The operation $\sharp$ is described by the action of the cycle $(2,3) \in \rS_4$ so that all orbits and thus $\Phi^{-1}\mrm{EuEx}(G_8)$ are stable under $\sharp$. A comparison with the data in \cite{Chlouveraki.MGAP-functions-for-th} shows that the essential hyperplanes for $G_8$ are precisely the 24 hyperplanes in the orbits $1$, $2$, $3$, and in the suborbit $\Sigma_{ \lbrace \lbrace 1,3 \rbrace, \lbrace 2,4 \rbrace \rbrace}.(1,-1,-1,1)$ of the orbit $5$. Hence, $\mrm{RouEx}(G_8) \subs \Phi^{-1}\mrm{EuEx}(G_8)$. Furthermore, the generic Rouquier families coincide with the generic Calogero--Moser families. This proves \ref{main_theorem_2} for $G_8$. \qed
\end{paran}

\subsection{$G_{10}$}

\begin{paran}
The group $G_{10}$ is of order 288 and can be realized as the matrix group over $\bbQ(\zeta)$, $\zeta \dopgleich \zeta_{12}$, generated by the reflections
\[
s \dopgleich \begin{pmatrix} 1 & 0 \\ 0 & \zeta^4 \end{pmatrix}, \quad t \dopgleich \frac{1}{3} \begin{pmatrix} \zeta^3 - \zeta^2 + \zeta + 2 & \frac{1}{2} ( \zeta^3 + 2\zeta^2 - 2\zeta - 1) \\ \zeta^3 + 2\zeta^2 - 2\zeta - 1 & 2\zeta^3 + \zeta^2 - \zeta + 1 \end{pmatrix}
\]
There are five conjugacy classes of reflections: the ones of $s$ and $s^2$ (both of order $3$ and length $8$), the ones of $t$ and $t^3$ (both of order $4$ and length $6$), and the one of $t^2$ (of order $2$ and length $6$). Hence, there are two orbits of reflection hyperplanes, namely the orbit $\Omega_1$ of the reflection hyperplane of $s$ and the orbit $\Omega_2$ of the reflection hyperplane of $t$. The parameters for the restricted rational Cherednik algebra are therefore $k \dopgleich (k_{1,0},k_{1,1},k_{1,2}, k_{2,0},k_{2,1},k_{2,2},k_{2,3})$ and we can now compute that
\[ 
\begin{array}{lcl}
c_k(s) & = & (-\zeta^2 - 1)k_{1,0} + (-\zeta^2 + 2)k_{1,1} + (2\zeta^2 - 1)k_{1,2}\\
c_k(s^2) & = & (\zeta^2 - 2)k_{1,0} + (\zeta^2 + 1)k_{1,1} + (-2\zeta^2 + 1)k_{1,2} \\
c_k(t) & = & (-\zeta^3 - 1)k_{2,0} + (-\zeta^3 + 1)k_{2,1} + (\zeta^3 + 1)k_{2,2} + (\zeta^3 - 1)k_{2,3} \\
c_k(t^2) & = & -2k_{2,0} + 2k_{2,1} - 2k_{2,2} + 2k_{2,3} \\
c_k(t^3) & = & (\zeta^3 - 1)k_{2,0} + (\zeta^3 + 1)k_{2,1} + (-\zeta^3 + 1)k_{2,2} + (-\zeta^3 - 1)k_{2,3} \;.
\end{array}
\]
The group $G_{10}$ has 48 irreducible characters and from the character table we can compute the table \ref{g10_data} on page \pageref{g10_data}.
\begin{table}[!htbp]
\centering
\scriptsize
\begin{tabular}{|c|c|c|c|c|c|c|c|}
\hline
$\lambda$ & $\lambda(s)$ & $\lambda(s^2)$ & $\lambda(t)$ & $\lambda(t^2)$ & $\lambda(t^3)$ & $\omega_{\lambda}(\check{\mrm{eu}}_k)$ & ss \\ \hline \hline
$\phi_{1,0}$& $1$ & $1$ & $1$ & $1$ & $1$ & $24k_{1,0} - 24k_{1,1} + 24k_{2,0} - 24k_{2,1}$  & n \\ \hline
$\phi_{1,6}$& $1$ & $1$ & $\zeta^{3}$ & $-1$ & $-\zeta^{3}$ & $24k_{1,0} - 24k_{1,1}$  & n \\ \hline
$\phi_{1,12}$& $1$ & $1$ & $-1$ & $1$ & $-1$ & $24k_{1,0} - 24k_{1,1} + 24k_{2,0} - 24k_{2,3}$  & n \\ \hline
$\phi_{1,18}$& $1$ & $1$ & $-\zeta^{3}$ & $-1$ & $\zeta^{3}$ & $24k_{1,0} - 24k_{1,1} + 24k_{2,0} - 24k_{2,2}$  & n \\ \hline
$\phi_{1,8}$& $\zeta^{2} - 1$ & $-\zeta^{2}$ & $1$ & $1$ & $1$ & $24k_{2,0} - 24k_{2,1}$  & n \\ \hline
$\phi_{1,14}$& $\zeta^{2} - 1$ & $-\zeta^{2}$ & $\zeta^{3}$ & $-1$ & $-\zeta^{3}$ & $0$  & n \\ \hline
$\phi_{1,20}$& $\zeta^{2} - 1$ & $-\zeta^{2}$ & $-1$ & $1$ & $-1$ & $24k_{2,0} - 24k_{2,3}$  & n \\ \hline
$\phi_{1,26}$& $\zeta^{2} - 1$ & $-\zeta^{2}$ & $-\zeta^{3}$ & $-1$ & $\zeta^{3}$ & $24k_{2,0} - 24k_{2,2}$  & n \\ \hline
$\phi_{1,16}$& $-\zeta^{2}$ & $\zeta^{2} - 1$ & $1$ & $1$ & $1$ & $24k_{1,0} - 24k_{1,2} + 24k_{2,0} - 24k_{2,1}$  & n \\ \hline
$\phi_{1,22}$& $-\zeta^{2}$ & $\zeta^{2} - 1$ & $\zeta^{3}$ & $-1$ & $-\zeta^{3}$ & $24k_{1,0} - 24k_{1,2}$  & n \\ \hline
$\phi_{1,28}$& $-\zeta^{2}$ & $\zeta^{2} - 1$ & $-1$ & $1$ & $-1$ & $24k_{1,0} - 24k_{1,2} + 24k_{2,0} - 24k_{2,3}$  & n \\ \hline
$\phi_{1,34}$& $-\zeta^{2}$ & $\zeta^{2} - 1$ & $-\zeta^{3}$ & $-1$ & $\zeta^{3}$ & $24k_{1,0} - 24k_{1,2} + 24k_{2,0} - 24k_{2,2}$  & n \\ \hline
$\phi_{2,9}$& $-1$ & $-1$ & $\zeta^{3} + 1$ & $0$ & $-\zeta^{3} + 1$ & $12k_{1,0} - 12k_{1,2} + 12k_{2,0} - 12k_{2,1}$  & n \\ \hline
$\phi_{2,12}$& $-1$ & $-1$ & $0$ & $2$ & $0$ & $12k_{1,0} - 12k_{1,2} + 24k_{2,0} - 12k_{2,1} - 12k_{2,3}$  & n \\ \hline
$\phi_{2,15}'$& $-1$ & $-1$ & $-\zeta^{3} + 1$ & $0$ & $\zeta^{3} + 1$ & $12k_{1,0} - 12k_{1,2} + 24k_{2,0} - 12k_{2,1} - 12k_{2,2}$  & n \\ \hline
$\phi_{2,15}''$& $-1$ & $-1$ & $\zeta^{3} - 1$ & $0$ & $-\zeta^{3} - 1$ & $12k_{1,0} - 12k_{1,2} + 12k_{2,0} - 12k_{2,3}$  & n \\ \hline
$\phi_{2,18}$& $-1$ & $-1$ & $0$ & $-2$ & $0$ & $12k_{1,0} - 12k_{1,2} + 12k_{2,0} - 12k_{2,2}$  & n \\ \hline
$\phi_{2,21}$& $-1$ & $-1$ & $-\zeta^{3} - 1$ & $0$ & $\zeta^{3} - 1$ & $12k_{1,0} - 12k_{1,2} + 24k_{2,0} - 12k_{2,2} - 12k_{2,3}$  & n \\ \hline
$\phi_{2,5}$& $-\zeta^{2} + 1$ & $\zeta^{2}$ & $\zeta^{3} + 1$ & $0$ & $-\zeta^{3} + 1$ & $24k_{1,0} - 12k_{1,1} - 12k_{1,2} + 12k_{2,0} - 12k_{2,1}$  & n \\ \hline
$\phi_{2,8}$& $-\zeta^{2} + 1$ & $\zeta^{2}$ & $0$ & $2$ & $0$ & $24k_{1,0} - 12k_{1,1} - 12k_{1,2} + 24k_{2,0} - 12k_{2,1} - 12k_{2,3}$  & n \\ \hline
$\phi_{2,11}'$& $-\zeta^{2} + 1$ & $\zeta^{2}$ & $-\zeta^{3} + 1$ & $0$ & $\zeta^{3} + 1$ & $24k_{1,0} - 12k_{1,1} - 12k_{1,2} + 24k_{2,0} - 12k_{2,1} - 12k_{2,2}$  & n \\ 
\hline
$\phi_{2,11}''$& $-\zeta^{2} + 1$ & $\zeta^{2}$ & $\zeta^{3} - 1$ & $0$ & $-\zeta^{3} - 1$ & $24k_{1,0} - 12k_{1,1} - 12k_{1,2} + 12k_{2,0} - 12k_{2,3}$  & n \\ \hline
$\phi_{2,14}$& $-\zeta^{2} + 1$ & $\zeta^{2}$ & $0$ & $-2$ & $0$ & $24k_{1,0} - 12k_{1,1} - 12k_{1,2} + 12k_{2,0} - 12k_{2,2}$  & n \\ \hline
$\phi_{2,17}$& $-\zeta^{2} + 1$ & $\zeta^{2}$ & $-\zeta^{3} - 1$ & $0$ & $\zeta^{3} - 1$ & $24k_{1,0} - 12k_{1,1} - 12k_{1,2} + 24k_{2,0} - 12k_{2,2} - 12k_{2,3}$  & n \\ 
\hline
$\phi_{2,1}$& $\zeta^{2}$ & $-\zeta^{2} + 1$ & $\zeta^{3} + 1$ & $0$ & $-\zeta^{3} + 1$ & $12k_{1,0} - 12k_{1,1} + 12k_{2,0} - 12k_{2,1}$  & n \\ \hline
$\phi_{2,4}$& $\zeta^{2}$ & $-\zeta^{2} + 1$ & $0$ & $2$ & $0$ & $12k_{1,0} - 12k_{1,1} + 24k_{2,0} - 12k_{2,1} - 12k_{2,3}$  & n \\ \hline
$\phi_{2,7}'$& $\zeta^{2}$ & $-\zeta^{2} + 1$ & $-\zeta^{3} + 1$ & $0$ & $\zeta^{3} + 1$ & $12k_{1,0} - 12k_{1,1} + 24k_{2,0} - 12k_{2,1} - 12k_{2,2}$  & n \\ \hline
$\phi_{2,7}''$& $\zeta^{2}$ & $-\zeta^{2} + 1$ & $\zeta^{3} - 1$ & $0$ & $-\zeta^{3} - 1$ & $12k_{1,0} - 12k_{1,1} + 12k_{2,0} - 12k_{2,3}$  & n \\ \hline
$\phi_{2,10}$& $\zeta^{2}$ & $-\zeta^{2} + 1$ & $0$ & $-2$ & $0$ & $12k_{1,0} - 12k_{1,1} + 12k_{2,0} - 12k_{2,2}$  & n \\ \hline
$\phi_{2,13}$& $\zeta^{2}$ & $-\zeta^{2} + 1$ & $-\zeta^{3} - 1$ & $0$ & $\zeta^{3} - 1$ & $12k_{1,0} - 12k_{1,1} + 24k_{2,0} - 12k_{2,2} - 12k_{2,3}$  & n \\ \hline
$\phi_{3,8}''$& $0$ & $0$ & $-1$ & $-1$ & $-1$ & $16k_{1,0} - 8k_{1,1} - 8k_{1,2} + 16k_{2,0} - 8k_{2,2} - 8k_{2,3}$  & y \\ \hline
$\phi_{3,14}$& $0$ & $0$ & $-\zeta^{3}$ & $1$ & $\zeta^{3}$ & $16k_{1,0} - 8k_{1,1} - 8k_{1,2} + 24k_{2,0} - 8k_{2,1} - 8k_{2,2} - 8k_{2,3}$  & y \\ \hline
$\phi_{3,8}'$& $0$ & $0$ & $1$ & $-1$ & $1$ & $16k_{1,0} - 8k_{1,1} - 8k_{1,2} + 16k_{2,0} - 8k_{2,1} - 8k_{2,2}$  & y \\ \hline
$\phi_{3,2}$& $0$ & $0$ & $\zeta^{3}$ & $1$ & $-\zeta^{3}$ & $16k_{1,0} - 8k_{1,1} - 8k_{1,2} + 16k_{2,0} - 8k_{2,1} - 8k_{2,3}$  & y \\ \hline
$\phi_{3,16}$& $0$ & $0$ & $-1$ & $-1$ & $-1$ & $16k_{1,0} - 8k_{1,1} - 8k_{1,2} + 16k_{2,0} - 8k_{2,2} - 8k_{2,3}$  & y \\ \hline
$\phi_{3,10}'$& $0$ & $0$ & $-\zeta^{3}$ & $1$ & $\zeta^{3}$ & $16k_{1,0} - 8k_{1,1} - 8k_{1,2} + 24k_{2,0} - 8k_{2,1} - 8k_{2,2} - 8k_{2,3}$  & y \\ \hline
$\phi_{3,4}$& $0$ & $0$ & $1$ & $-1$ & $1$ & $16k_{1,0} - 8k_{1,1} - 8k_{1,2} + 16k_{2,0} - 8k_{2,1} - 8k_{2,2}$  & y \\ \hline
$\phi_{3,10}''$& $0$ & $0$ & $\zeta^{3}$ & $1$ & $-\zeta^{3}$ & $16k_{1,0} - 8k_{1,1} - 8k_{1,2} + 16k_{2,0} - 8k_{2,1} - 8k_{2,3}$  & y \\ \hline
$\phi_{3,12}'$& $0$ & $0$ & $-1$ & $-1$ & $-1$ & $16k_{1,0} - 8k_{1,1} - 8k_{1,2} + 16k_{2,0} - 8k_{2,2} - 8k_{2,3}$  & y \\ \hline
$\phi_{3,6}''$& $0$ & $0$ & $-\zeta^{3}$ & $1$ & $\zeta^{3}$ & $16k_{1,0} - 8k_{1,1} - 8k_{1,2} + 24k_{2,0} - 8k_{2,1} - 8k_{2,2} - 8k_{2,3}$  & y \\ \hline
$\phi_{3,12}''$& $0$ & $0$ & $1$ & $-1$ & $1$ & $16k_{1,0} - 8k_{1,1} - 8k_{1,2} + 16k_{2,0} - 8k_{2,1} - 8k_{2,2}$  & y \\ \hline
$\phi_{3,6}'$& $0$ & $0$ & $\zeta^{3}$ & $1$ & $-\zeta^{3}$ & $16k_{1,0} - 8k_{1,1} - 8k_{1,2} + 16k_{2,0} - 8k_{2,1} - 8k_{2,3}$  & y \\ \hline
$\phi_{4,9}$& $1$ & $1$ & $0$ & $0$ & $0$ & $18k_{1,0} - 12k_{1,1} - 6k_{1,2} + 18k_{2,0} - 6k_{2,1} - 6k_{2,2} - 6k_{2,3}$  & n \\ \hline
$\phi_{4,11}$& $\zeta^{2} - 1$ & $-\zeta^{2}$ & $0$ & $0$ & $0$ & $12k_{1,0} - 6k_{1,1} - 6k_{1,2} + 18k_{2,0} - 6k_{2,1} - 6k_{2,2} - 6k_{2,3}$  & n \\ \hline
$\phi_{4,7}$& $-\zeta^{2}$ & $\zeta^{2} - 1$ & $0$ & $0$ & $0$ & $18k_{1,0} - 6k_{1,1} - 12k_{1,2} + 18k_{2,0} - 6k_{2,1} - 6k_{2,2} - 6k_{2,3}$  & y \\ \hline
$\phi_{4,3}$& $1$ & $1$ & $0$ & $0$ & $0$ & $18k_{1,0} - 12k_{1,1} - 6k_{1,2} + 18k_{2,0} - 6k_{2,1} - 6k_{2,2} - 6k_{2,3}$  & y \\ \hline
$\phi_{4,5}$& $\zeta^{2} - 1$ & $-\zeta^{2}$ & $0$ & $0$ & $0$ & $12k_{1,0} - 6k_{1,1} - 6k_{1,2} + 18k_{2,0} - 6k_{2,1} - 6k_{2,2} - 6k_{2,3}$  & y \\ \hline
$\phi_{4,13}$& $-\zeta^{2}$ & $\zeta^{2} - 1$ & $0$ & $0$ & $0$ & $18k_{1,0} - 6k_{1,1} - 12k_{1,2} + 18k_{2,0} - 6k_{2,1} - 6k_{2,2} - 6k_{2,3}$  & n \\ \hline
\end{tabular} \caption{Character data for $G_{10}$.} \label{g10_data}
\end{table}
From this table we can deduce that the only non-singleton generic Euler families are
\[
\begin{array}{c}
\lbrace \phi_{3,8}'',\phi_{3,16},\phi_{3,12}'\rbrace, \lbrace \phi_{3,14},\phi_{3,10}',\phi_{3,6}''\rbrace, \lbrace \phi_{3,8}',\phi_{3,4},\phi_{3,12}''\rbrace, \lbrace 
\phi_{3,2},\phi_{3,10}'',\phi_{3,6}'\rbrace, \\
\lbrace \phi_{4,9},\phi_{4,3}\rbrace, \lbrace \phi_{4,11},\phi_{4,5}\rbrace, \lbrace \phi_{4,7},\phi_{4,13}\rbrace \;.
\end{array}
\]
A close look at the supersingularity column in table \ref{g10_data} now shows that all these families are good and therefore already Calogero--Moser $k$-families for all $k \notin \Phi^{-1}\mrm{EuEx}(G_{10})$ by \ref{trick}. Hence,
\[
\mrm{Eu}_k = \mrm{CM}_k \quad \tn{for all } \quad k \notin \Phi^{-1}\mrm{EuEx}(G_{10}) \;.
\]

From \ref{g10_data} we can also compute that $\Phi^{-1}\mrm{EuEx}(G_{10}$) is the union of the following 300 hyperplanes, which form 12 orbits under the Young subgroup $\Sigma \dopgleich \Sigma_{\lbrace \lbrace 1,2,3 \rbrace, \lbrace 4,5,6,7 \rbrace \rbrace}$ of $\rS_7$:

\begin{table}[H]
\centering
\footnotesize
\begin{minipage}[b]{0.45\linewidth}
\centering
\begin{tabular}{|l|l|l|}
\hline
Label & Orbit & Length \\
\hline \hline
$1$a & $\Sigma.(0, 0, 0, 0, 0, 1, -1)$ & $6$ \\
\hline
$1$b & $\Sigma.(0, 1, -1, 0, 0, 0, 0)$ & $3$ \\
\hline
$2$a & $\Sigma.(0, 1, -1, 0, 0, -1, 1)$ & $36$ \\
\hline
$2$b & $\Sigma.(0, 0, 0, 1, -1, -1, 1)$ & $3$ \\
\hline
$3$a & $\Sigma.(0, 1, -1, 0, -2, 1, 1)$ & $72$ \\
\hline
$3$b & $\Sigma.(1, -2, 1, 0, 0, -1, 1)$ & $36$ \\
\hline
\end{tabular}
\end{minipage}
\begin{minipage}[b]{0.45\linewidth}
\centering
\begin{tabular}{|l|l|l|}
\hline
Label & Orbit & Length \\
\hline \hline

$4$ & $\Sigma.(0, 1, -1, -1, -1, 1, 1)$ & $18$ \\
\hline
$5$ & $\Sigma.(1, -3, 2, -3, 1, 1, 1)$ & $24$ \\
\hline
$6$ & $\Sigma.(1, -2, 1, 0, -2, 1, 1)$ & $36$ \\
\hline
$7$ & $\Sigma.(1, -2, 1, -3, 1, 1, 1)$ & $12$ \\
\hline
$8$ & $\Sigma.(1, -2, 1, -2, -2, 1, 3)$ & $36$ \\
\hline
$9$ & $\Sigma.(1, -2, 1, -1, -1, 1, 1)$ & $18$ \\
\hline
\end{tabular} 
\end{minipage}
\caption{$\Phi^{-1}\mrm{EuEx}(G_{10})$.} 
\end{table}

\noindent The operation $\sharp$ is described by the action of the permutation $(2,3)(5,6) \in \Sigma$ so that all orbits and thus $\Phi^{-1}\mrm{EuEx}(G_{10})$ are stable under $\sharp$. A comparison with the data in \cite{Chlouveraki.MGAP-functions-for-th} shows that the essential hyperplanes of $G_{10}$ are precisely the 81 hyperplanes in the orbits $1$a, $1$b, $6$, $7$, the suborbit $\wt{\Sigma}.(0, 1, -1, 0, 0, -1, 1)$ of orbit $2$a, and the suborbit $\wt{\Sigma}.((0, 1, -1, -1, -1, 1, 1)$ of orbit $4$, where 
\[
\wt{\Sigma} \dopgleich \langle (1,3),(5,7), (1,3,2), (4,6)(5,7) \rangle \leq \Sigma \;.
\]
This shows that $\mrm{RouEx}(G_{10}) \subs \Phi^{-1}\mrm{EuEx}(G_{10})$. Furthermore, the generic Rouquier families coincide with the generic Calogero--Moser families. This proves \ref{main_theorem_2} for $G_{10}$.
\end{paran}

\subsection{$G_{23}$}

\begin{paran}
The group $G_{23}$ is of order 120 and is actually the Coxeter group of type $H_3$. It can be realized as the matrix group over $\bbQ(\zeta)$, where $\zeta \dopgleich \zeta_5$, generated by the reflections
\[
s \dopgleich \begin{pmatrix} -1 & 0 & 0 \\ \tau & 1 & 0 \\ 0 & 0 & 1 \end{pmatrix}, \quad t \dopgleich \begin{pmatrix} 1 & \tau & 0 \\ 0 & -1 & 0 \\ 0 & 1 & 1 \end{pmatrix}, \quad u \dopgleich \begin{pmatrix} 1 & 0 & 0 \\ 0 & 1 & 1 \\ 0 & 0 & -1 \end{pmatrix} \;,
\]
where $\tau \dopgleich -\zeta^3 - \zeta^2$. Note that $(2\tau - 1)^2 = 5$ and so $\tau \in \bbR$. There is just one conjugacy class of reflections, namely the one of $s$ (of order $2$ and length $15$). Hence, there is just one orbit $\Omega_1$ of reflection hyperplanes and so the parameters for the restricted rational Cherednik algebra are $k \dopgleich (k_{1,0}, k_{1,1})$. We can now compute that
\[ %
c_k(s) = -2k_{1,0} + 2k_{1,1} \;.
\]
The group $G_{23}$ has 10 irreducible characters and from the character table we can compute table \ref{g23_data}. 
\begin{table}[H]   
\centering 
\footnotesize
\begin{minipage}[b]{0.36\linewidth}
\centering 
\begin{tabular}{|c|c|c|c|}
\hline
$\lambda$ & $\lambda(s)$ & $\omega_{\lambda}(\check{\mrm{eu}}_k)$ & ss \\ \hline \hline
$\phi_{1,15}$& $-1$ & $0$  & n \\ \hline
$\phi_{1,0}$& $1$ & $30k_{1,0} - 30k_{1,1}$  & n \\ \hline
$\phi_{5,5}$& $-1$ & $12k_{1,0} - 12k_{1,1}$  & n \\ \hline
$\phi_{5,2}$& $1$ & $18k_{1,0} - 18k_{1,1}$  & n \\ \hline
$\phi_{3,6}$& $-1$ & $10k_{1,0} - 10k_{1,1}$  & y \\ \hline
\end{tabular}
\end{minipage}
\begin{minipage}[b]{0.36\linewidth}
\centering 
\begin{tabular}{|c|c|c|c|}
\hline
$\lambda$ & $\lambda(s)$ & $\omega_{\lambda}(\check{\mrm{eu}}_k)$ & ss \\ \hline \hline
$\phi_{3,8}$& $-1$ & $10k_{1,0} - 10k_{1,1}$  & n \\ \hline
$\phi_{3,1}$& $1$ & $20k_{1,0} - 20k_{1,1}$  & y \\ \hline
$\phi_{3,3}$& $1$ & $20k_{1,0} - 20k_{1,1}$  & n \\ \hline
$\phi_{4,3}$& $0$ & $15k_{1,0} - 15k_{1,1}$  & y \\ \hline
$\phi_{4,4}$& $0$ & $15k_{1,0} - 15k_{1,1}$  & y \\ \hline
\end{tabular} 
\end{minipage} \caption{Character data for $G_{23}$.} \label{g23_data}
\end{table}
\noindent From this table we deduce that the the non-singleton generic Euler families are
\[
\lbrace \phi_{3,6},\phi_{3,8}\rbrace, \lbrace \phi_{3,1},\phi_{3,3}\rbrace, \lbrace \phi_{4,3},\phi_{4,4}\rbrace 
\]
and that all these families are good so that they are already Calogero--Moser $k$-families for all $k \notin \Phi^{-1}\mrm{EuEx}(G_{23})$. Hence,
\[
\mrm{Eu}_k = \mrm{CM}_k \quad \tn{for all } \quad k \notin \Phi^{-1}\mrm{EuEx}(G_{23}) \;.
\]

From table \ref{g23_data} we can also compute that $\Phi^{-1}\mrm{EuEx}(G_{23}$) just consists of the single hyperplane defined by 
\[
k_{1,0} - k_{1,1} \;.
\]
This hyperplane is clearly stable under the operation $\sharp$. A comparison with the data in \cite{Chlouveraki.MGAP-functions-for-th} shows that this is also precisely the essential hyperplane for $G_{23}$ and that the generic Rouquier families coincide with the the generic Calogero--Moser families. This proves \ref{main_theorem_2} for $G_{23}$. 
\end{paran}

\subsection{$G_{24}$}

\begin{paran}
The group $G_{24}$ is of order 336 and can be realized as the matrix group over $\bbQ(\zeta)$, where $\zeta \dopgleich \zeta_7$, generated by the reflections

\[ 
s \dopgleich \begin{pmatrix} -1 & 1 & \tau \\ 0 & 1 & 0 \\ 0 & 0 & 1 \end{pmatrix}, \quad t \dopgleich \begin{pmatrix} 1 & 0 & 0 \\ 1 & -1 & 1 \\ 0 & 0 & 1 \end{pmatrix}, \quad u \dopgleich \begin{pmatrix} 1 & 0 & 0 \\ 0 & 1 & 0 \\ -\tau-1 & 1 & -1 \end{pmatrix} \;,
\]
where $\tau \dopgleich \zeta^4 + \zeta^2 + \zeta$.
We note that we have not given the same realization as in \cite{CHEVIE-JM-4} as this one contains a lot of denominators and is thus not nice to write down. This different choice causes no troubles however as all irreducible characters of $G_{24}$ are already uniquely determined by their $(d,b)$-pair. 
 
There is just one conjugacy class of reflections, namely the one of $s$ (of order $2$ and length $21$) and so there is just one orbit $\Omega_1$ of reflection hyperplanes. The parameters for the restricted rational Cherednik algebra are therefore $k \dopgleich (k_{1,0},k_{1,1})$ and we can now compute that
\[
c_k(s) = -2k_{1,0} + 2k_{1,1} \;.
\]
The group $G_{26}$ has 12 irreducible characters and from the character table we can compute the table \ref{g24_data}.
\begin{table}[H] 
\centering
\footnotesize
\begin{minipage}[b]{0.36\linewidth}
\centering
\begin{tabular}{|c|c|c|c|}
\hline
$\lambda$ & $\lambda(s)$ & $\omega_{\lambda}(\check{\mrm{eu}}_k)$ & ss\\ \hline \hline
$\phi_{1,0}$& $1$ & $42k_{1,0} - 42k_{1,1}$  & n \\ \hline
$\phi_{1,21}$& $-1$ & $0$  & n \\ \hline
$\phi_{3,10}$& $-1$ & $14k_{1,0} - 14k_{1,1}$  & y \\ \hline
$\phi_{3,3}$& $1$ & $28k_{1,0} - 28k_{1,1}$  & y \\ \hline
$\phi_{3,1}$& $1$ & $28k_{1,0} - 28k_{1,1}$  & y \\ \hline
$\phi_{3,8}$& $-1$ & $14k_{1,0} - 14k_{1,1}$  & y \\ \hline
\end{tabular}
\end{minipage}
\begin{minipage}[b]{0.36\linewidth}
\centering
\begin{tabular}{|c|c|c|c|}
\hline
$\lambda$ & $\lambda(s)$ & $\omega_{\lambda}(\check{\mrm{eu}}_k)$ & ss\\ \hline \hline
$\phi_{6,2}$& $2$ & $28k_{1,0} - 28k_{1,1}$  & y \\ \hline
$\phi_{6,9}$& $-2$ & $14k_{1,0} - 14k_{1,1}$  & y \\ \hline
$\phi_{7,3}$& $1$ & $24k_{1,0} - 24k_{1,1}$  & n \\ \hline
$\phi_{7,6}$& $-1$ & $18k_{1,0} - 18k_{1,1}$  & n \\ \hline
$\phi_{8,5}$& $0$ & $21k_{1,0} - 21k_{1,1}$  & y \\ \hline
$\phi_{8,4}$& $0$ & $21k_{1,0} - 21k_{1,1}$  & y \\ \hline
\end{tabular}
\end{minipage}
\caption{Character data for $G_{24}$.} \label{g24_data}
\end{table}
\noindent From this table we see that the non-singleton generic Euler families are
\[
\lbrace \phi_{3,8},\phi_{3,10},\phi_{6,9} \rbrace, \lbrace \phi_{3,1},\phi_{3,3}, \phi_{6,2} \rbrace,  \lbrace \phi_{8,4} , \phi_{8,5} \rbrace 
\]
and that all these families are good. Hence, 
\[
\mrm{Eu}_k = \mrm{CM}_k \quad \tn{for all} \quad k \notin \Phi^{-1}\mrm{EuEx}(G_{24}) \;.
\]

From table \ref{g24_data} we can also easily compute that $\Phi^{-1}\mrm{EuEx}(G_{24})$ just consists of the single hyperplane defined by
\[
k_{1,0} - k_{1,1} \;.
\]
This hyperplane is clearly stable under the operation $\sharp$. A comparison with the data in \cite{Chlouveraki.MGAP-functions-for-th} shows that this is also the unique essential hyperplane for $G_{24}$ and that the generic Rouquier families coincide with the just determined generic Calogero--Moser families. This proves \ref{main_theorem_2} for $G_{24}$. \qed
\end{paran}

\subsection{$G_{26}$} \label{g26}

\begin{paran}
The group $G_{26}$ is of order 1296 and can be realized as the matrix group over $\bbQ(\zeta)$, where $\zeta \dopgleich \zeta_3$, generated by the reflections

\[
s \dopgleich \begin{pmatrix} 
1 & 0 & 0 \\ 
0 & 0 & 1 \\ 
0 & 1 & 0
\end{pmatrix}, \quad
t \dopgleich \begin{pmatrix} 
1 & 0 & 0 \\ 
0 & 1 & 0 \\ 
0 & 0 & \zeta
\end{pmatrix}, \quad
u \dopgleich \frac{1}{3} \begin{pmatrix} 
\zeta + 2 & \zeta-1 & \zeta-1 \\
\zeta-1 & \zeta+2 & \zeta-1 \\
\zeta-1 & \zeta-1 & \zeta+2
\end{pmatrix} \;.
\]
There are three conjugacy classes of reflections: the one of $s$ (of order $2$ and length $9$), and the ones of $t$ and $t^2$ (both of order $3$ and length $12$). 
Hence, there are two orbits of reflection hyperplanes, namely the orbit $\Omega_1$ of the reflection hyperplane of $s$ and the orbit $\Omega_2$ of the reflection hyperplane of $t$. The parameters for the restricted rational Cherednik algebra are therefore $k\dopgleich(k_{1,0},k_{1,1},k_{2,0},k_{2,1},k_{2,2})$ and we can now compute that
\[ 
\begin{array}{lcl}
c_k(s) & = & 2k_{1,0} + 2k_{1,1} \;, \\
c_k(t) & = & (-\zeta - 2)k_{2,0} + (-\zeta + 1)k_{2,1} + (2\zeta + 1)k_{2,2} \;, \\
c_k(t^2) & = & (\zeta - 1)k_{2,0} + (\zeta + 2)k_{2,1} + (-2\zeta - 1)k_{2,2} \;.
\end{array}
\]
The group $G_{26}$ has 48 irreducible characters and from the character table we can compute the table \ref{g26_data} on page \pageref{g26_data}.
\begin{table}[!htbp] 
\centering
\footnotesize
\begin{tabular}{|c|c|c|c|c|c|}
\hline
$\lambda$ & $\lambda(s)$ & $\lambda(t)$ & $\lambda(t^2)$ & $\omega_{\lambda}(\check{\mrm{eu}}_k)$ & ss \\ \hline \hline
$\phi_{1,0}$& $1$ & $1$ & $1$ & $18k_{1,0} - 18k_{1,1} + 36k_{2,0} - 36k_{2,1}$  & n \\ \hline
$\phi_{1,9}$& $-1$ & $1$ & $1$ & $36k_{2,0} - 36k_{2,1}$  & n \\ \hline
$\phi_{1,33}$& $-1$ & $-\zeta - 1$ & $\zeta$ & $36k_{2,0} - 36k_{2,2}$  & n \\ \hline
$\phi_{1,21}$& $-1$ & $\zeta$ & $-\zeta - 1$ & $0$  & n \\ \hline
$\phi_{1,24}$& $1$ & $-\zeta - 1$ & $\zeta$ & $18k_{1,0} - 18k_{1,1} + 36k_{2,0} - 36k_{2,2}$  & n \\ \hline
$\phi_{1,12}$& $1$ & $\zeta$ & $-\zeta - 1$ & $18k_{1,0} - 18k_{1,1}$  & n \\ \hline
$\phi_{2,24}$& $-2$ & $-1$ & $-1$ & $18k_{2,0} - 18k_{2,2}$  & n \\ \hline
$\phi_{2,15}$& $2$ & $-1$ & $-1$ & $18k_{1,0} - 18k_{1,1} + 18k_{2,0} - 18k_{2,2}$  & n \\ \hline
$\phi_{2,12}$& $-2$ & $\zeta + 1$ & $-\zeta$ & $18k_{2,0} - 18k_{2,1}$  & n \\ \hline
$\phi_{2,3}$& $2$ & $\zeta + 1$ & $-\zeta$ & $18k_{1,0} - 18k_{1,1} + 18k_{2,0} - 18k_{2,1}$  & n \\ \hline
$\phi_{2,18}$& $-2$ & $-\zeta$ & $\zeta + 1$ & $36k_{2,0} - 18k_{2,1} - 18k_{2,2}$  & n \\ \hline
$\phi_{2,9}$& $2$ & $-\zeta$ & $\zeta + 1$ & $18k_{1,0} - 18k_{1,1} + 36k_{2,0} - 18k_{2,1} - 18k_{2,2}$  & n \\ \hline
$\phi_{3,6}$& $3$ & $0$ & $0$ & $18k_{1,0} - 18k_{1,1} + 24k_{2,0} - 12k_{2,1} - 12k_{2,2}$  & n \\ \hline
$\phi_{3,15}$& $-3$ & $0$ & $0$ & $24k_{2,0} - 12k_{2,1} - 12k_{2,2}$  & n \\ \hline
$\phi_{3,8}''$& $-1$ & $2\zeta + 1$ & $-2\zeta - 1$ & $6k_{1,0} - 6k_{1,1} + 12k_{2,0} - 12k_{2,1}$  & n \\ \hline
$\phi_{3,5}''$& $1$ & $2\zeta + 1$ & $-2\zeta - 1$ & $12k_{1,0} - 12k_{1,1} + 12k_{2,0} - 12k_{2,1}$  & n \\ \hline
$\phi_{3,8}'$& $-1$ & $-\zeta + 1$ & $\zeta + 2$ & $6k_{1,0} - 6k_{1,1} + 36k_{2,0} - 24k_{2,1} - 12k_{2,2}$  & n \\ \hline
$\phi_{3,5}'$& $1$ & $-\zeta + 1$ & $\zeta + 2$ & $12k_{1,0} - 12k_{1,1} + 36k_{2,0} - 24k_{2,1} - 12k_{2,2}$  & n \\ \hline
$\phi_{3,20}$& $-1$ & $-\zeta - 2$ & $\zeta - 1$ & $6k_{1,0} - 6k_{1,1} + 24k_{2,0} - 24k_{2,2}$  & n \\ \hline
$\phi_{3,17}$& $1$ & $-\zeta - 2$ & $\zeta - 1$ & $12k_{1,0} - 12k_{1,1} + 24k_{2,0} - 24k_{2,2}$  & n \\ \hline
$\phi_{3,16}''$& $-1$ & $-2\zeta - 1$ & $2\zeta + 1$ & $6k_{1,0} - 6k_{1,1} + 36k_{2,0} - 12k_{2,1} - 24k_{2,2}$  & n \\ \hline
$\phi_{3,13}''$& $1$ & $-2\zeta - 1$ & $2\zeta + 1$ & $12k_{1,0} - 12k_{1,1} + 36k_{2,0} - 12k_{2,1} - 24k_{2,2}$  & n \\ \hline
$\phi_{3,4}$& $-1$ & $\zeta + 2$ & $-\zeta + 1$ & $6k_{1,0} - 6k_{1,1} + 24k_{2,0} - 24k_{2,1}$  & n \\ \hline
$\phi_{3,1}$& $1$ & $\zeta + 2$ & $-\zeta + 1$ & $12k_{1,0} - 12k_{1,1} + 24k_{2,0} - 24k_{2,1}$  & n \\ \hline
$\phi_{3,16}'$& $-1$ & $\zeta - 1$ & $-\zeta - 2$ & $6k_{1,0} - 6k_{1,1} + 12k_{2,0} - 12k_{2,2}$  & n \\ \hline
$\phi_{3,13}'$& $1$ & $\zeta - 1$ & $-\zeta - 2$ & $12k_{1,0} - 12k_{1,1} + 12k_{2,0} - 12k_{2,2}$  & n \\ \hline
$\phi_{6,8}''$& $2$ & $-2\zeta - 1$ & $2\zeta + 1$ & $12k_{1,0} - 12k_{1,1} + 30k_{2,0} - 12k_{2,1} - 18k_{2,2}$  & n \\ \hline
$\phi_{6,11}''$& $-2$ & $-2\zeta - 1$ & $2\zeta + 1$ & $6k_{1,0} - 6k_{1,1} + 30k_{2,0} - 12k_{2,1} - 18k_{2,2}$  & n \\ \hline
$\phi_{6,8}'$& $2$ & $\zeta - 1$ & $-\zeta - 2$ & $12k_{1,0} - 12k_{1,1} + 18k_{2,0} - 6k_{2,1} - 12k_{2,2}$  & n \\ \hline
$\phi_{6,11}'$& $-2$ & $\zeta - 1$ & $-\zeta - 2$ & $6k_{1,0} - 6k_{1,1} + 18k_{2,0} - 6k_{2,1} - 12k_{2,2}$  & n \\ \hline
$\phi_{6,2}$& $2$ & $\zeta + 2$ & $-\zeta + 1$ & $12k_{1,0} - 12k_{1,1} + 24k_{2,0} - 18k_{2,1} - 6k_{2,2}$  & n \\ \hline
$\phi_{6,5}$& $-2$ & $\zeta + 2$ & $-\zeta + 1$ & $6k_{1,0} - 6k_{1,1} + 24k_{2,0} - 18k_{2,1} - 6k_{2,2}$  & n \\ \hline
$\phi_{6,4}''$& $2$ & $2\zeta + 1$ & $-2\zeta - 1$ & $12k_{1,0} - 12k_{1,1} + 18k_{2,0} - 12k_{2,1} - 6k_{2,2}$  & n \\ \hline
$\phi_{6,7}''$& $-2$ & $2\zeta + 1$ & $-2\zeta - 1$ & $6k_{1,0} - 6k_{1,1} + 18k_{2,0} - 12k_{2,1} - 6k_{2,2}$  & n \\ \hline
$\phi_{6,10}$& $2$ & $-\zeta - 2$ & $\zeta - 1$ & $12k_{1,0} - 12k_{1,1} + 24k_{2,0} - 6k_{2,1} - 18k_{2,2}$  & n \\ \hline
$\phi_{6,13}$& $-2$ & $-\zeta - 2$ & $\zeta - 1$ & $6k_{1,0} - 6k_{1,1} + 24k_{2,0} - 6k_{2,1} - 18k_{2,2}$  & n \\ \hline
$\phi_{6,4}'$& $2$ & $-\zeta + 1$ & $\zeta + 2$ & $12k_{1,0} - 12k_{1,1} + 30k_{2,0} - 18k_{2,1} - 12k_{2,2}$  & n \\ \hline
$\phi_{6,7}'$& $-2$ & $-\zeta + 1$ & $\zeta + 2$ & $6k_{1,0} - 6k_{1,1} + 30k_{2,0} - 18k_{2,1} - 12k_{2,2}$  & n \\ \hline
$\phi_{8,6}'$& $0$ & $2$ & $2$ & $9k_{1,0} - 9k_{1,1} + 27k_{2,0} - 18k_{2,1} - 9k_{2,2}$  & y \\ \hline
$\phi_{8,3}$& $0$ & $2$ & $2$ & $9k_{1,0} - 9k_{1,1} + 27k_{2,0} - 18k_{2,1} - 9k_{2,2}$  & y \\ \hline
$\phi_{8,9}''$& $0$ & $-2\zeta - 2$ & $2\zeta$ & $9k_{1,0} - 9k_{1,1} + 27k_{2,0} - 9k_{2,1} - 18k_{2,2}$  & y \\ \hline
$\phi_{8,12}$& $0$ & $-2\zeta - 2$ & $2\zeta$ & $9k_{1,0} - 9k_{1,1} + 27k_{2,0} - 9k_{2,1} - 18k_{2,2}$  & y \\ \hline
$\phi_{8,6}''$& $0$ & $2\zeta$ & $-2\zeta - 2$ & $9k_{1,0} - 9k_{1,1} + 18k_{2,0} - 9k_{2,1} - 9k_{2,2}$  & y \\ \hline
$\phi_{8,9}'$& $0$ & $2\zeta$ & $-2\zeta - 2$ & $9k_{1,0} - 9k_{1,1} + 18k_{2,0} - 9k_{2,1} - 9k_{2,2}$  & y \\ \hline
$\phi_{9,8}$& $-3$ & $0$ & $0$ & $6k_{1,0} - 6k_{1,1} + 24k_{2,0} - 12k_{2,1} - 12k_{2,2}$  & y \\ \hline
$\phi_{9,5}$& $3$ & $0$ & $0$ & $12k_{1,0} - 12k_{1,1} + 24k_{2,0} - 12k_{2,1} - 12k_{2,2}$  & y \\ \hline
$\phi_{9,10}$& $-3$ & $0$ & $0$ & $6k_{1,0} - 6k_{1,1} + 24k_{2,0} - 12k_{2,1} - 12k_{2,2}$  & y \\ \hline
$\phi_{9,7}$& $3$ & $0$ & $0$ & $12k_{1,0} - 12k_{1,1} + 24k_{2,0} - 12k_{2,1} - 12k_{2,2}$  & y \\ \hline
\end{tabular}
\caption{Character data for $G_{26}$.} \label{g26_data}
\end{table}
From this table we deduce that the non-singleton generic Euler families are
\[
\lbrace \phi_{8,6}',\phi_{8,3}\rbrace, \lbrace 
\phi_{8,9}'',\phi_{8,12}\rbrace, \lbrace \phi_{8,6}'',\phi_{8,9}'\rbrace, \lbrace \phi_{9,8},\phi_{9,10}\rbrace, \lbrace \phi_{9,5},\phi_{9,7}\rbrace 
\]
and that all these families are good. Hence,
\[
\mrm{Eu}_k = \mrm{CM}_k \quad \tn{for all } \quad k \notin \Phi^{-1}\mrm{EuEx}(G_{26}) \;.
\]

From table \ref{g26_data} we can furthermore compute that $\Phi^{-1}\mrm{EuEx}(G_{26})$ is the union of the 169 hyperplanes listed in table \ref{g26_planes} on page \pageref{g26_planes}, which form 22 orbits under the Young subgroup $\Sigma \dopgleich \Sigma_{ \lbrace \lbrace 1,2 \rbrace, \lbrace 3,4,5\rbrace \rbrace}$ of $\rS_5$.
\begin{table}[H] \footnotesize
\centering
\begin{minipage}[b]{0.38\linewidth}
\centering
\begin{tabular}{|l|l|l|}
\hline
Label & Orbit & Length \\
\hline \hline
1a & $\Sigma.(0, 0, 0, 1, -1)$ & $3$ \\ \hline
1b & $\Sigma.(1, -1, 0, 0, 0)$ & $1$ \\ \hline
2 & $\Sigma.(0, 0, 1, -3, 2)$ & $6$ \\ \hline
3& $\Sigma.(0, 0, 1, -2, 1)$ & $3$ \\ \hline
4 & $\Sigma.(1, -1, 0, -4, 4)$ & $6$ \\ \hline
5 & $\Sigma.(1, -1, 0, -3, 3)$ & $6$ \\ \hline
6a & $\Sigma.(1, -1, 0, -2, 2)$ & $6$ \\ \hline
6b & $\Sigma.(2, -2, 0, -1, 1)$ & $6$ \\ \hline
7 & $\Sigma.(1, -1, 0, -1, 1)$ & $6$ \\ \hline
8 & $\Sigma.(1, -1, -6, 1, 5)$ & $12$ \\ \hline
9 & $\Sigma.(1, -1, -6, 2, 4)$ & $12$ \\ \hline
\end{tabular}
\end{minipage}
\begin{minipage}[b]{0.38\linewidth}
\centering
\begin{tabular}{|l|l|l|}
\hline
Label & Orbit & Length \\
\hline \hline
10 & $\Sigma.(1, -1, -5, 2, 3)$ & $12$ \\ \hline
11 & $\Sigma.(1, -1, -4, 1, 3)$ & $12$ \\ \hline
12 & $\Sigma.(1, -1, -4, 2, 2)$ & $6$ \\ \hline
13 & $\Sigma.(1, -1, -3, 1, 2)$ & $12$ \\ \hline
14 & $\Sigma.(1, -1, -2, 1, 1)$ & $6$ \\ \hline
15 & $\Sigma.(2, -2, -5, 2, 3)$ & $12$ \\ \hline
16 & $\Sigma.(2, -2, -4, 1, 3)$ & $12$ \\ \hline
17 & $\Sigma.(2, -2, -3, 1, 2)$ & $12$ \\ \hline
18 & $\Sigma.(2, -2, -2, 1, 1)$ & $6$ \\ \hline
19 & $\Sigma.(3, -3, -4, 2, 2)$ & $6$ \\ \hline
20 & $\Sigma.(3, -3, -2, 1, 1)$ & $6$ \\ \hline
\end{tabular}
\end{minipage} \caption{$\Phi^{-1}\mrm{EuEx}(G_{26})$. Here, $\Sigma \dopgleich \Sigma_{ \lbrace \lbrace 1,2 \rbrace, \lbrace 3,4,5\rbrace \rbrace}$.} \label{g26_planes}
\end{table}
The operation $\sharp$ is described by the action of the cycle $(4,5) \in \Sigma$ so that all orbits and thus $\Phi^{-1}(\mrm{EuEx}(G_{26})$ are stable under $\sharp$. 
A comparison with the data in \cite{Chlouveraki.MGAP-functions-for-th} shows that the 31 hyperplanes in the orbits $1$a, $1$b, $3$, $7$, $13$, and $14$, are precisely the essential hyperplanes of $G_{26}$ and that the generic Rouquier families coincide with the generic Calogero--Moser families. This proves \ref{main_theorem_2} for $G_{26}$. \qed\end{paran}

\section{Conclusions}

\begin{paran} \label{euler_coarser}
For all the exceptional groups 
\[
W \in \lbrace G_4,G_5,G_6,G_8,G_{10}, G_{23}=H_3, G_{24},G_{25},G_{26} \rbrace
\]
considered in \S3 we have proven that 
\[
\mrm{Eu}_k = \mrm{CM}_k \quad \tn{for all} \quad k \notin \Phi^{-1}\mrm{EuEx}(W) \;.
\]
As we covered quite a variety of groups, this might---if we are very optimistic---suggest that this equality always holds. This is (un)fortunately not true however as we will see for the symmetric group $W \dopgleich \rS_6$.\footnote{I would like to thank Maria Chlouveraki for pointing me towards this example.} To be able to write down the critical characters also in terms of partitions (and not only in the form $\phi_{d,b})$ we have to choose one of the two reflection representations of $W$ corresponding to the partitions $(5,1)$ and $(2,2,2)$ of $6$, which are interchanged by the outer automorphism of $\rS_6$. We choose  the first one and now consider the reflection group $\Gamma$ defined in this way. The reflections in $\Gamma$ are then precisely the transpositions in $W$. These form a single conjugacy class of order $2$ and so there is just one orbit $\Omega_1$ of reflection hyperplanes. Therefore the variables for the Cherednik algebra of $\Gamma$ are $k \dopgleich (k_{1,0}, k_{1,1})$ and we can now compute that
\[
c_k(s) = -2k_{1,0} + 2k_{1,1} \;,
\]
where $s$ is any reflection in $\Gamma$. The group $W$ has 11 irreducible characters and from the character table of $W$ we can easily compute table \ref{s6_data}. 
\begin{table}[H] 
\centering
\footnotesize
\begin{minipage}[t]{0.50\linewidth}
\centering
\begin{tabular}[t]{|c|c|c|}
\hline
$\lambda$ & $\lambda(s)$ & $\omega_{\lambda}(\check{\mrm{eu}}_k)$ \\ \hline \hline
$\phi_{1,1} = (6)$ & $1$ & $30k_{1,0} - 30k_{1,1}$ \\ \hline
$\phi_{1,15} = (1,1,1,1,1,1)$ & $-1$ & $0$  \\ \hline
$\phi_{5,3} = (3,3)$ & $1$ & $18k_{1,0} - 18k_{1,1}$ \\\hline
$\phi_{5,1} = (5,1)$ & $3$ & $24k_{1,0} - 24k_{1,1}$ \\\hline
$\phi_{5,6} = (2,2,2)$ & $-1$ & $12k_{1,0} - 12k_{1,1}$ \\\hline
$\phi_{5,10} = (2,1,1,1,1)$ & $-3$ & $6k_{1,0} - 6k_{1,1}$ \\\hline
\end{tabular}
\end{minipage}
\begin{minipage}[t]{0.45\linewidth}
\centering
\begin{tabular}[t]{|c|c|c|}
\hline
$\lambda$ & $\lambda(s)$ & $\omega_{\lambda}(\check{\mrm{eu}}_k)$ \\ \hline \hline
$\phi_{9,2} = (4,2)$ & $3$ & $20k_{1,0} - 20k_{1,1}$ \\\hline
$\phi_{9,7} = (2,2,1,1)$ & $-3$ & $10k_{1,0} - 10k_{1,1}$ \\\hline
$\phi_{10,3} = (4,1,1)$ & $2$ & $18k_{1,0} - 18k_{1,1}$ \\\hline
$\phi_{10,6} = (3,1,1,1)$ & $-2$ & $12k_{1,0} - 12k_{1,1}$ \\\hline
$\phi_{10,4}  =(3,2,1)$ & $0$ & $15k_{1,0} - 15k_{1,1}$ \\ \hline
\end{tabular}
\end{minipage}
\caption{Data for $\rS_6$.} \label{s6_data}
\end{table}
\noindent From this table we see that there are two non-singleton generic Euler families, namely 
$\lbrace \phi_{5,3}, \phi_{10,3} \rbrace$ and $\lbrace \phi_{5,6}, \phi_{10,6} \rbrace$. However, it follows from \cite{GorMar-Calogero-Moser-space-rest-0} that the generic Calogero--Moser families for $\rS_6$ are singletons and so $\mrm{CM}_k$ is strictly finer than $\mrm{Eu}_k$ for generic parameters $k$.

We note that the smallest example in the series $G(m,1,n)$ where this occurs is actually the group $G(2,1,4) = B_4$, which is of order 384. 
\end{paran}

\begin{question}
Our discussion leads to the following questions:
\begin{enum_thm}
\item What is the actual relationship between generic Calogero--Moser families, generic Rouquier families and generic Euler families?
\item What is the abstract explanation for the failure of Martino's generic parameter conjecture for the group $G_{25}$? What is so special about this group?
\item What is the relationship between $\mrm{EuEx}(\Gamma)$ and $\mrm{CMEx}(\Gamma)$. Is $\mrm{CMEx}(\Gamma)$ contained in $\mrm{EuEx}(\Gamma)$?
\item What is the relationship between $(\Phi^{-1}\mrm{CMGen}(\Gamma))^\sharp$ and $\mrm{RouGen}(\Gamma)$? Are they equal?
\item \textit{All} results so far about Martino's conjecture (including those for the infinite series) are obtained by determining the Calogero--Moser families and then comparing them with the Rouquier families. Is there a \textit{structural} explanation for the positive results on Martino's conjecture?
\end{enum_thm} 
\end{question}

\thispagestyle{pagenumonly}

\ifthenelse{\boolean{printbib}}
{
    \ifthenelse{\boolean{book}}
    {
        \cleardoublepage
        \phantomsection
    }
    {}
    \ifthenelse{\boolean{lang_ger}}
    {
        \ifthenelse{\boolean{book}}
        {
            \addcontentsline{toc}{chapter}{Literatur}
        }
        {}
    }
    {}
    \ifthenelse{\boolean{lang_eng}}
    {
        \ifthenelse{\boolean{book}}
        {
            \addcontentsline{toc}{chapter}{Bibliography}
        }
        {}
    }
    {}
    \pagestyle{pagenumonly}
    \printbibliography
    \ifthenelse{\boolean{book}}
    {
        \pagestyle{chapteronly}
    }
    {}
   
}
{}
\ifthenelse{\boolean{printbiburlremark}}
{
    \noindent \small \textbf{Remark.} Whenever possible, URLs to electronic versions (preferably freely available preprints) of the references are given. Note however, that these might differ from the published versions.
}
{}

\ifthenelse{\boolean{makeidx}}
{
    \pagestyle{empty}
    \ifthenelse{\boolean{book}}
    {
    	\cleardoublepage
        \phantomsection
        \addcontentsline{toc}{chapter}{Index}
    }
    {}
    \ifthenelse{\boolean{article}}
    {
    	\clearpage
    	\phantomsection
        \addcontentsline{toc}{section}{\textbf{Index}}
    }
    {}
    \printindex
    \pagestyle{chapteronly}
}{}

\thispagestyle{pagenumonly}

\end{document}